\documentclass[11pt]{amsart}
\usepackage{amsmath,amssymb,epsfig,rotating}

\usepackage[all]{xy}
\usepackage{pinlabel}

\usepackage[
  pdftitle={Symmetries and adjunction inequalities for knot Floer homology},
  pdfauthor={Bijan Sahamie},
  citecolor=blue,
  colorlinks,
  linkcolor=blue,
  urlcolor=red,
  ]{hyperref}

\title{Symmetries and adjunction inequalities for knot Floer homology}

\author{Bijan Sahamie} 
\address{Mathematisches Institut der LMU M\"unchen, 
Theresienstrasse 39, 80333 M\"unchen Germany}
\email{sahamie@math.lmu.de}
\urladdr{http://www.math.lmu.de/~sahamie}

\theoremstyle{plain} 
\newtheorem{theorem}{Theorem}[section]   
\newtheorem{lem}[theorem]{Lemma}         
\newtheorem{prop}[theorem]{Proposition}
\newtheorem{cor}[theorem]{Corollary}

\theoremstyle{definition}
\newtheorem{definition}[theorem]{Definition}   
\newtheorem{rem}{Remark}

\newcommand{\Z}{\mathbb{Z}}
\newcommand{\R}{\mathbb{R}}
\newcommand{\C}{\mathbb{C}}
\newcommand{\Q}{\mathbb{Q}}

\newcommand{\delhatw}{\widehat{\partial}^w}
\newcommand{\fgo}{\mathbb{F}_g^0}

\newcommand{\fge}{\mathbb{F}_g^e}
\newcommand{\fgepo}{\mathbb{F}_g^{e+1}}
\newcommand{\foe}{\mathbb{F}_0^e}

\newcommand{\afat}{{\boldsymbol\alpha}}
\newcommand{\bfat}{{\boldsymbol\beta}}
\newcommand{\gfat}{{\boldsymbol\gamma}}
\newcommand{\xfat}{{\boldsymbol x}}
\newcommand{\yfat}{{\boldsymbol y}}
\newcommand{\zfat}{{\boldsymbol z}}

\numberwithin{equation}{section}

\begin{document}
%\nocite{Kuperberg02}
%\nocite{Ginzburg}
%\nocite{Ginzburg02}
\newcommand{\sone}{\mathbb{S}^1}
\newcommand{\lra}{\longrightarrow}
\newcommand{\lmt}{\longmapsto}
%
%\newcommand{\son}{\mbox{\rm SO}_n}

%
% Standard spaces/sets.
%
\newcommand{\ztwo}{\mathbb{Z}_2}
\newcommand{\RP}{\mbox{\rm RP}}
\newcommand{\spinc}{\mbox{\rm Spin}^c}
\def\co{\colon\thinspace}
\newcommand{\stwo}{\mathbb{S}^2}
\newcommand{\pd}{\text{PD}}
\newcommand{\sprung}{\\[0.3cm]}
\newcommand{\fund}{\pi_1}
\newcommand{\kerg}{\mbox{\rm Ker}_G\,}
\def\Ker#1{\mbox{\rm Ker}_{#1}\,}
\newcommand{\im}{\mbox{\rm Im}\,}
\newcommand{\id}{\mbox{\rm id}}
\newcommand{\sothree}{\mathbb{SO}_3}
\newcommand{\sthree}{\mathbb{S}^{3}}
\newcommand{\disc}{\mbox{\rm D}}
\newcommand{\systwo}{C^{\infty}(Y,\stwo)}
\newcommand{\inner}{\text{int}}
\newcommand{\bk}{\backslash}

%
% Contact geometric definitions.
%
\newcommand{\contstand}{(\mathbb{R}^3,\xi_0)}
\newcommand{\cont}{(M,\xi)}
\newcommand{\sym}{\xi_{sym}}
\newcommand{\xistd}{\xi_{std}}
\newcommand{\cyl}{\mbox{\rm Cyl}_{r_0}^\mu}
\newcommand{\stap}{\mbox{\rm S}_+}
\newcommand{\stam}{\mbox{\rm S}_-}
\newcommand{\stapm}{\mbox{\rm S}_\pm}
\newcommand{\modulo}{\;\;\mbox{\rm mod}\,}

%
% Morse-homology
%
\newcommand{\crit}{\mbox{\rm Crit}}
\newcommand{\MC}{\mbox{\rm MC}}
\newcommand{\bmorse}{\partial^{\mbox{\rm \begin{tiny}M\!C\end{tiny}}}}
\newcommand{\M}{\mathcal{M}}
\newcommand{\stable}{\mbox{\rm W}^s}
\newcommand{\unstable}{\mbox{\rm W}^u}
\newcommand{\ind}{\mbox{\rm ind}}
\newcommand{\Mhat}{\widehat{\M}}
\newcommand{\What}{\widehat{W}}
\newcommand{\modphi}{\mathcal{M}_\phi}
%
%
%

%
% Heegaard-Floer
%
\newcommand{\moduli}{\mathcal{M}_{J_s}(x,y)}
\newcommand{\modulit}{\mathcal{M}_{J_{s,t}}}
\newcommand{\modulittau}{\mathcal{M}_{J_{s,t(\tau)}}}
\newcommand{\modulittaubig}{\mathcal{M}_{J_{s,t,\tau}}}
\newcommand{\modhat}{\widehat{\mathcal{M}}_{J_s}(x,y)}
\newcommand{\modulihat}{\widehat{\mathcal{M}}}
\newcommand{\modhatone}{\widehat{\mathcal{M}}_{J_{s,1}}}
\newcommand{\modhatzero}{\widehat{\mathcal{M}}_{J_{s,0}}}
\newcommand{\modhatphi}{\widehat{\mathcal{M}}_{\phi}}
\newcommand{\moduliiso}{\mathcal{M}^{\Psi_t}}
\newcommand{\modspace}{\mathcal{M}}
\newcommand{\phihat}{\widehat{\phi}}
\newcommand{\Phiinfty}{\Phi^\infty}
\newcommand{\Dhat}{\widehat{\D}}
\newcommand{\cops}{\partial_{J_s}}
\newcommand{\talpha}{\mathbb{T}_{\boldsymbol{\alpha}}}
\newcommand{\tbeta}{\mathbb{T}_{\boldsymbol{\beta}}}
\newcommand{\tgamma}{\mathbb{T}_{\boldsymbol{\gamma}}}
\def\marge#1{\marginpar{\scriptsize{#1}}}
\def\br#1{\begin{rotate}{90}#1\end{rotate}}
\newcommand{\hfhat}{\widehat{\mbox{\rm HF}}}
\newcommand{\sfh}{\mbox{\rm SFH}}
\newcommand{\sbottom}{\underline{s}}
\newcommand{\tbottom}{\underline{t}}
\newcommand{\cfhat}{\widehat{\mbox{\rm CF}}}
\newcommand{\cfinfty}{\mbox{\rm CF}^\infty}
\def\cfbb#1{\mbox{\rm CF}^+_{\leq #1}}
\def\cfbo#1{\mbox{\rm CF}^-_{\geq -#1}}
\newcommand{\cfleq}{\mbox{\rm CF}^{\leq 0}}
\newcommand{\cfcirc}{\mbox{\rm CF}^\circ}
\newcommand{\cfkinfty}{\mbox{\rm CFK}^{\infty}}
\newcommand{\cfkhat}{\widehat{\mbox{\rm CFK}}}
\newcommand{\cfkminus}{\mbox{\rm CFK}^{-}}
\newcommand{\cfkpstar}{\mbox{\rm CFK}^{+,*}}
\newcommand{\cfkostar}{\mbox{\rm CFK}^{0,*}}
\newcommand{\hfkcirc}{\mbox{\rm HFK}^\circ}
\newcommand{\hfkhat}{\widehat{\mbox{\rm HFK}}}
\newcommand{\hfkplus}{\mbox{\rm HFK}^+}
\newcommand{\hfkminus}{\mbox{\rm HFK}^-}

\newcommand{\cfkoc}{\mbox{\rm CFK}^{\bullet,\circ}}
\newcommand{\hfkoc}{\mbox{\rm HFK}^{\bullet,\circ}}
\newcommand{\cfkom}{\mbox{\rm CFK}^{\bullet,-}}
\newcommand{\hfkom}{\mbox{\rm HFK}^{\bullet,-}}
\newcommand{\hfkmo}{\mbox{\rm HFK}^{-,\bullet}}
\newcommand{\hfkpo}{\mbox{\rm HFK}^{+,\bullet}}
\newcommand{\hfkio}{\mbox{\rm HFK}^{\infty,\bullet}}
\newcommand{\hfkoo}{\mbox{\rm HFK}^{\bullet,\bullet}}
\newcommand{\cfkop}{\mbox{\rm CFK}^{\bullet,+}}
\newcommand{\hfkop}{\mbox{\rm HFK}^{\bullet,+}}
\newcommand{\hfkoi}{\mbox{\rm HFK}^{\bullet,\infty}}
\newcommand{\Foo}{F^{\bullet,\bullet}}
\newcommand{\Foc}{F^{\bullet,\circ}}
\newcommand{\foc}{f^{\bullet,\circ}}
\newcommand{\Fco}{F^{\circ,\bullet}}
\newcommand{\Fom}{F^{\bullet,-}}
\newcommand{\Foi}{F^{\bullet,\infty}}
\newcommand{\mMhat}{\widehat{\mathcal{M}}}
\newcommand{\cfoo}{f^{\bullet,\bullet}}

\newcommand{\fhat}{\widehat{f}}

\newcommand{\cfkco}{\mbox{\rm CFK}^{\circ,\bullet}}

\newcommand{\Goo}{G^{\bullet,\bullet}}
\newcommand{\Hoo}{H^{\bullet,\bullet}}

\def\cfinftyfilt#1#2{\mbox{\rm CFK}^{#1,#2}}
\newcommand{\hfinfty}{\mbox{\rm HF}^\infty}
\newcommand{\hfinftwist}{\underline{\mbox{\rm {HF}}^\infty}}
\newcommand{\lfh}{\widehat{\mbox{\rm HFL}}}
\newcommand{\fcirc}{f^\circ}
\newcommand{\chat}{\widehat{c}}
\newcommand{\Fhat}{\widehat{F}}
\newcommand{\Jhat}{\widehat{J}}
\newcommand{\That}{\widehat{T}}
\newcommand{\Hhat}{\widehat{H}}
\newcommand{\Fcirc}{F^\circ}
\newcommand{\hattheta}{\widehat{\Theta}}
\newcommand{\shattheta}{\widehat{\theta}}
\newcommand{\cfminus}{\mbox{\rm CF}^-}
\newcommand{\hfminus}{\mbox{\rm HF}^-}
\newcommand{\cfplus}{\mbox{\rm CF}^+}
\newcommand{\hfplus}{\mbox{\rm HF}^+}
\newcommand{\hfcirc}{\mbox{\rm HF}^\circ}
\def\hfbb#1{\hfplus_{\leq #1}}
\def\hfbo#1{\hfminus_{\geq -#1}}
\newcommand{\Hs}{\mathcal{H}_s}
\newcommand{\mh}{\mathcal{H}}
\newcommand{\gr}{\mbox{gr}}
\newcommand{\parinfty}{\partial^\infty}
\newcommand{\parhat}{\widehat{\partial}}
\newcommand{\parplus}{\partial^+}
\newcommand{\parminus}{\partial^-}
\newcommand{\symg}{\mbox{\rm Sym}^g(\Sigma)}
\newcommand{\symgg}{\mbox{\rm Sym}^{2g}(\Sigma)}
\newcommand{\symc}{\mbox{\rm Sym}^g(\mathbb{C})}
\newcommand{\pitwo}{\pi_2}
\newcommand{\pitwoham}{\pi_2^{\Psi_t}}
\newcommand{\symcon}{\mbox{\rm Sym}^g(\Sigma_1\#\Sigma_2)}
\newcommand{\symgone}{\mbox{\rm Sym}^{g_1}(\Sigma_1)}
\newcommand{\symgtwo}{\mbox{\rm Sym}^{g_2}(\Sigma_2)}
\newcommand{\symgmo}{\mbox{\rm Sym}^{g-1}(\Sigma)}
\newcommand{\symggmo}{\mbox{\rm Sym}^{2g-1}(\Sigma)}
\newcommand{\dom}{\mathcal{D}}
\newcommand{\bigtrans}{\left.\bigcap\hspace{-0.27cm}\right|\hspace{0.1cm}}
\newcommand{\tlt}{\times\ldots\times}
\newcommand{\Isotopy}{\Gamma^\infty_{\Psi_t}}
\newcommand{\orient}{\mathnormal{o}}
\newcommand{\ob}{\mathnormal{ob}}
\newcommand{\SL}{\mbox{\rm SL}}
\newcommand{\rhotilde}{\widetilde{\rho}}
\newcommand{\domstar}{\dom_*}
\newcommand{\domststar}{\dom_{**}}
\newcommand{\betaprime}{\beta'}
\newcommand{\betapp}{\beta''}
\newcommand{\betatilde}{\widetilde{\beta}}
\newcommand{\deltaprime}{\delta'}
\newcommand{\tbetaprime}{\mathbb{T}_{\beta'}}
\newcommand{\talphaprime}{\mathbb{T}_{\alpha'}}
\newcommand{\tdelta}{\mathbb{T}_{\delta}}
\newcommand{\phidelta}{\phi^{\Delta}}
\newcommand{\domtilde}{\widetilde{\dom}}
\newcommand{\loss}{\widehat{\mathcal{L}}}
\newcommand{\bargamma}{\overline{\Gamma}}
\newcommand{\alphaprime}{\alpha'}
\newcommand{\ga}{\Gamma_{\alpha;\beta',\beta''}}
\newcommand{\gbone}{\Gamma_{\alpha;\beta,\widetilde{\beta}}^{w,1}}
\newcommand{\gbtwo}{\Gamma_{\alpha;\beta,\widetilde{\beta}}^{w,2}}
\newcommand{\gbthree}{\Gamma_{\alpha;\beta,\widetilde{\beta}}^{w,3}}
\newcommand{\gbfour}{\Gamma_{\alpha;\beta,\widetilde{\beta}}^{w,4}}
\newcommand{\gcone}{\Gamma_{\alpha;\delta,\delta'}^{w,1}}
\newcommand{\gctwo}{\Gamma_{\alpha;\delta,\delta'}^{w,2}}
\newcommand{\gcthree}{\Gamma_{\alpha;\delta,\delta'}^{w,3}}
\newcommand{\gcfour}{\Gamma_{\alpha;\delta,\delta'}^{w,4}}

\newcommand{\eab}{\epsilon_{\alpha\beta}}
\newcommand{\ead}{\epsilon_{\alpha\delta}}
\newcommand{\hqhat}{\widehat{\mbox{\rm HQ}}}
\newcommand{\cupb}{\cup_\partial}
\newcommand{\oa}{\overline{a}}
\newcommand{\ab}{\alpha\beta}
\newcommand{\ad}{\alpha\delta}
\newcommand{\adb}{\alpha\delta\beta}
\newcommand{\tila}{\widetilde{a}}
\newcommand{\tilb}{\widetilde{b}}
\def\pdehn#1#2{D_{#1}^{+,#2}} 
\def\ndehn#1#2{D_{#1}^{-,#2}} 
\newcommand{\fraks}{\mathfrak{s}}
\newcommand{\frakt}{\mathfrak{t}}
\newcommand{\Ghat}{\widehat{G}}
\newcommand{\pointswap}{\phi^{\rm PS}}
\newcommand{\mA}{\mathcal{A}}
\newcommand{\mJ}{\mathcal{J}}
\newcommand{\mM}{\mathcal{M}}
\newcommand{\mN}{\mathcal{N}}
\newcommand{\mH}{\mathcal{H}}
\newcommand{\sh}{\mbox{\rm h}}
\newcommand{\shb}{\mbox{\rm \underline{h}}}

%
% BündelZeugs
%
\newcommand{\bund}{\mathcal{P}}
\newcommand{\diag}{\Delta^{\!\!E}}
\newcommand{\inter}{m_{\diag}}
\newcommand{\ozs}{Ozsv\'{a}th}
\newcommand{\sza}{Szab\'{o}}
\newcommand{\mS}{\mathcal{S}}
\newcommand{\oA}{\overline{A}}
\begin{abstract}
We derive symmetries and adjunction inequalities of the knot Floer 
homology groups which appear to be especially interesting for 
homologically essential knots. Furthermore, we obtain an adjunction
inequality for cobordism maps in knot Floer homologies. We demonstrate 
the adjunction inequalities and symmetries in explicit calculations 
which recover some of the main results from \cite{eftek} on 
longitude Floer homology and also give rise to vanishing results 
on knot Floer homologies. Furthermore, using symmetries we prove 
that the knot Floer homology of a fiber distinguishes 
$\stwo\times\sone$ from other $\sone$-bundles over surfaces.
\end{abstract}
\maketitle
\section{Introduction}\label{parone}
\noindent Heegaard Floer homology was introduced by Peter Ozsv\'ath and 
Zoltan Szab\'o in \cite{OsZa01} (see~\cite{Saha02} for a detailed 
introduction) and has turned out to be a useful tool in the study of 
low-dimensional topology. They also defined variants of this 
homology theory which are topological invariants of a pair $(Y,K)$ where $Y$ is 
a closed, oriented $3$-manifold and $K\subset Y$ a null-homologous knot 
(see \cite{OsZa04}). 
Knot Floer homologies have also proven to be very useful in knot theoretic
applications, the filtration on these groups carrying a lot of geometric 
information. In \cite{Saha}, we made the observation that the knot Floer 
homology is not restricted to homologically trivial knots. For $[K]=0$ the
knot theoretic information was especially encoded into a filtration constructed 
using a Seifert surface of $K$. In the case $[K]\not=0$ this filtration gets lost. 
The information given by the filtration, however, do not seem to get lost (at least
not fully), but are shifted into the $\spinc$-refinements of the groups. 
We find it natural to also study the groups for $[K]\not=0$. The first step 
in this study is to provide
tools making the homology groups accessible to computations. Recall, that 
computations of Heegaard Floer homologies are usually done
using surgery exact triangles and adjunction inequalities of the groups 
involved and of the maps induced by cobordisms. The groups $\hfkhat$ may 
be equipped with an adjunction inequality coming from sutured Floer 
homology (see~\cite[Theorem 2]{AJU2}). Furthermore, Juh\'{a}sz's work
on cobordisms maps for sutured Floer homologies provide a notion of cobordism
maps for the $\hfkhat$-case, as well (see~\cite{JuCOB}). In \cite{Saha04}, we 
introduced cobordism maps for various versions of knot Floer 
homology, especially for $\hfkoc$ (i.e.~$\hfkhat$, $\hfkom$ and $\hfkoi$, see~\S\ref{knotfloerhomology} for a definition).\vspace{0.3cm}\\
In this article, we study symmetry properties of knot Floer homology 
groups $\hfkoc$, adjunction inequalities for $\hfkoc$ and adjunction 
inequalities for cobordism maps of these groups. After discussing these 
concepts, we present some implications
of these results which are meant as a demonstration how these techniques may 
be applied in computations. We have to point out that the computational 
results we present (for the $\hfkhat$-homology) can be alternatively derived 
(and some strengthened) from work of Friedl, Juh\'{a}sz and Rasmussen on sutured 
Floer homology (see~\cite[Proposition~7.7]{FJR}).

\subsection{Adjunction inequalities}
We prove the following result. For the $\hfkhat$-case, cf.~also 
\cite[Theorem 2]{AJU2}.
\begin{theorem}\label{myadjuncineq} Let $Y$ be 
a closed, oriented $3$-manifold with knot $K$ and $F\subset Y$ a closed, oriented 
surface with genus $g(F)>0$ such that $\#(K\cap F)\leq1$ and $\#(F,K)\leq0$. 
 \begin{enumerate}
\item[(i)] If $K$ and $F$ intersect, the non-vanishing of $\hfkoc(Y,K;\fraks)$ 
implies that
\[
  -\bigl<c_1(\fraks),[F]\bigr>\leq 2g(F)-2.
\]
\item[(ii)] If $K$ and $F$ are disjoint, the non-vanishing of $\hfkoc(Y,K;\fraks)$ 
implies that
\[
  \bigl|\bigl<c_1(\fraks),[F]\bigr>\bigr|\leq 2g(F)-2.
\]
\end{enumerate}
\end{theorem}
The main difficulty here is that the surface $F$ and the knot $K$ 
might intersect: To elaborate a little on this problem, we would 
like to note that the adjunction inequalities are set up by 
constructing a special Heegaard diagram which is adapted to the 
surface $F$ (see~Proposition~\ref{mydiagram}). In such diagrams, 
we obtain a formula for computing the chern class of a 
$\spinc$-structure evaluated on $F$ in terms of information 
encoded in the Heegaard diagram (see~Lemma~\ref{chernclassformula}). 
However, to make this construction work in the presence of a 
knot, we have to simultaneously adapt the Heegaard diagram to both 
the knot and the surface. It turns out that there are basically 
two cases, namely the two given in the theorem. Multiple 
intersections have to be resolved by adding handles to the 
surface $F$. A detailed explanation of such a procedure is given 
in the proof of Proposition~\ref{thm:van}. To us it seemed peculiar 
that in case $(i)$ of the theorem the chern class term appears 
without absolute values. One might think that it should be possible 
to prove this result with absolute values. However, what we did 
seems to be the best that can be done. We are able to provide a 
counterexample to the estimate in $(i)$ with absolute values. We 
point the interested reader to Remark~\ref{remark01}. 

Additionally, we prove the following adjunction inequality for 
cobordism maps between knot Floer homologies. Here, $\Foc$ stands 
for $\Foo$, $\Fom$ or $\Foi$ (see~\cite{Saha04}).
%\vspace{0.3cm}\\
%
%
\begin{theorem}\label{mymapadjunc} Let $Y$ be a closed, oriented 
$3$-manifold with knot $K'$ and let $K$ be a homologically trivial 
knot disjoint from $K'$. Denote by $W$ the knot cobordism induced 
by surgery along a knot $K'$. Suppose we are given a closed, 
oriented surface $F\subset W$ with positive genus, 
$\#(I\times K'\cap F)\leq1$ and $\#(F,I\times K')\leq 0$. For 
$\fraks\in\spinc(W)$ we have the following result:
\begin{enumerate}
 \item[(i)] If $I\times K'$ and $F$ intersect, then $[F]^2\geq0$ 
 together with
 \[-\bigl<c_1(\fraks),[F]\bigr>-[F]^2>2g(F)-2\]
 implies that $\Foc_{K,\fraks}=0$.
 \item[(ii)] If $I\times K$ and $F$ are disjoint, then $[F]^2\geq0$ 
 together with
 \[\bigl|\bigl<c_1(\fraks),[F]\bigr>\bigr|+[F]^2>2g(F)-2\]
 implies that $\Foc_{K,\fraks}=0$.
\end{enumerate}
\end{theorem}

\subsection{Symmetries} In \S\ref{symmetries} we study symmetries 
of the knot Floer homologies. The groups $\hfkoc$ share a 
conjugation symmetry which is given in 
Proposition~\ref{conjuginvar}. The knot Floer homology 
$\hfkhat$ has some additional symmetries.
\begin{prop}\label{mysymmetry} Let $Y$ be a closed, oriented 
$3$-manifold with knot $K$ which is homologically essential. Then 
there is a canonical isomorphism
\[
  \pointswap\co
  \hfkhat(Y,K;\fraks)
  \lra
  \hfkhat(Y,-K;\fraks+PD[K])
\]
we will call {\bf point-swap isomorphism}.
\end{prop}
Using the homology class of $[K]$ we define a map $\mA$ on the 
$\spinc$-structures of $Y$ given by $\mA(\fraks)=\fraks+PD[K]$ for $\fraks\in\spinc(Y)$.
Using the conjugation map and the map $\mA$ we define a new conjugation map we 
call {\bf knot conjugation}
\[
 \mathcal{N}
 \co
 \spinc(Y)
 \lra\spinc(Y)
\]
given by $\mN=\mJ\circ\mA$. By applying point-swap symmetry and conjugation symmetry we obtain the following result.
\begin{cor}\label{mysymmetry2} Let $Y$ be a closed, oriented $3$-manifold with knot $K\subset Y$. Then for all $\fraks\in\spinc(Y)$ there is an isomorphism
\[
  \mN_{(Y,K);\fraks}\co
  \hfkhat(Y,K;\fraks)
  \overset{\cong}{\lra}
  \hfkhat(Y,K;\mN(\fraks)).
\]
\end{cor}
This symmetry turns out to be interesting, especially in light of $(i)$ of Theorem~\ref{myadjuncineq}. The knot conjugation symmetry shows that in case 
$[K]\not=0$ we have a {\it shifted} (or broken) symmetry in $PD[K]$-direction.
The missing absolute values in the adjunction inequalities, thus, can be thought of
as a manifestation of this {\it shifted} (or broken) symmetry.
The knot conjugation symmetry carries over to maps induced by cobordisms. 
\begin{prop}\label{knotconjug} Let $Y$ be a closed, oriented $3$-manifold 
and $K\subset Y$ a knot. Denote by $W$ a cobordism between $(Y,K)$ and 
$(Y',K')$. Then, there is a map
\[
 \mN\co\spinc(W)\lra\spinc(W)
\]
with the property that $\mN\circ\mN={\rm id}$ such that
\[
 \Foo_{W,\fraks}
 =
 \mN_{(Y',K')}
 \circ
 \Foo_{W,\mN(\fraks)}
 \circ
 \mN_{(Y,K)}
\]
for every $\fraks\in\spinc(W)$. We call this {\bf knot conjugation symmetry}.
Furthermore, the maps $\Foc_{W,\fraks}$ fulfill a conjugation 
symmetry which says that $\Foc_{W,\fraks}$ 
equals $\mJ_{Y'}\circ \Foc_{W,\mJ(\fraks)}\circ\mJ_{Y}$.
\end{prop}
The map $\mN$ on $\spinc(W)$ will be not be specified, entirely. In the 
proof of this proposition, we show that $\mN$ is a combination of conjugation 
and a shift with a constant class. However, the discussion in the proof gives 
all information needed to compute this shift.
\subsection{Calculations I} An immediate implication of knot conjugation symmetry
is the following result we prove in \S\ref{app01}.
\begin{theorem}\label{mysym:kfh} Let $Y$ be a closed, oriented $3$-manifold 
with a knot $K\subset Y$ whose associated homology class $[K]$ cannot 
be divided by two. Then the rank of the knot Floer homology of the pair 
$(Y,K)$ is even.
\end{theorem}
A combination of knot conjugation symmetry and the adjunction inequalities
gives the following result.
\begin{prop}\label{thm:van} Let $K,L\subset\sthree$ be arbitrary knots 
and $\Sigma_S$ a Seifert surface of $K$ with minimal genus $sg(K)$. Then 
the group
\[
 \hfkhat(\sthree_0(K),L;\fraks)
\]
\begin{itemize}
 \item[(a)] vanishes for $\spinc$-structures $\fraks$ outside of
\[
 \Bigl[-sg(K)-\lfloor\frac{\#(L\cap\Sigma_S)}{2}\rfloor+1,\,
sg(K)+\lfloor\frac{\#(L\cap\Sigma_S)}{2}\rfloor-1-lk(K,L)\Bigr]
\] 
if $\#(L\cap\Sigma_S)$ is non-zero, odd and $lk(K,L)\not=\#(L\cap\Sigma_S)$, or
if the number $\#(L\cap\Sigma_S)$ is even (or zero) and $lk(K,L)\geq0$,

\item[(b)] vanishes for $\spinc$-structures $\fraks$ outside of
\[
\Bigl[
-sg(K)-\bigl\lfloor\frac{\#(L\cap\Sigma_S)}{2}\bigr\rfloor+1-lk(K,L),\,
sg(K)+\lfloor\frac{\#(L\cap\Sigma_S)}{2}\rfloor-1\Bigr]
\]
if $\#(L\cap\Sigma_S)$ is non-zero, odd and $lk(K,L)=\#(L,\Sigma_S)$, 
or if the number $\#(L\cap\Sigma_S)$ is even (or zero) and $lk(K,L)\leq 0$.
\end{itemize}
\end{prop}
We gave this result for knots $K$ and $L$ in the $3$-sphere. However, there is an
immediate analogue for knots $K$ and $L$ in an arbitrary closed, oriented $3$-manifold.
In that case, $K$ has to be nullhomologous, but $L$ can be arbitrary. As a special
case of this theorem, in combination with knot conjugation symmetry, we almost recover
\cite[Theorem~3.2]{eftek} on longitude Floer homology.
\begin{cor}[cf.~Theorem~1.1.~of \cite{eftek}]\label{myres:lfh} For a 
knot $K\subset \sthree$ denote by $sg(K)$ its Seifert genus. 
The longitude Floer homology $\widehat{\mbox{\rm HFL}}(K,\fraks)$ vanishes for 
$\fraks> sg(K)$ in $1/2+\Z$ and for $\fraks<-sg(K)$ in $1/2+\Z$. Furthermore, 
we have that
\[
 \lfh(K,\fraks)\cong\lfh(K,-\fraks).
\]
\end{cor}
Eftekhary's Theorem~3.2 on longitude Floer homology consists of four 
statements of whom we recover three.

\subsection{Calculations II} 
When bringing homologically non-trivial knots into the picture the 
behavior of the knot Floer homologies change: By a result of Ozsv\'{a}th and 
Szab\'{o} the Heegaard Floer homology $\hfhat(Y)$ of every closed, oriented 
$3$-manifold is non-zero. If $K$ is a homologically trivial knot in $Y$, there 
is a spectral sequence from
the knot Floer homology $\hfkhat(Y,K)$ converging to $\hfhat(Y)$. In consequence,
for $K$ homologically trivial, the associated knot Floer homology is non-zero.
In contrast to this, in \cite{Saha} we discovered (and used implicitly) that $\stwo\times\sone$ admits a homologically essential knot whose knot Floer homology 
is completely zero.
\begin{prop}[see~Proof of Theorem 7.4~of~\cite{Saha}]\label{oldres} Denote 
by $K^*$ a fiber of the $\sone$-bundle $\stwo\times\sone$. The associated 
knot Floer homology of the pair $(\stwo\times\sone,K^*)$ 
vanishes, i.e.~the group $\hfkhat(\stwo\times\sone,K^*)$ is zero.
\end{prop}
Ozsv\'{a}th and Szab\'{o} show in \cite{OsZa05} that for a given contact 
manifold $(Y,\xi)$ there is a contact geometric invariant 
$\chat(\xi)\in\hfhat(-Y)$. Indeed, this invariant is an obstruction to 
overtwistedness of $\xi$ and it is particularly powerful as demonstrated 
by results of Lisca and Stipsicz (see~\cite{Stip02,Stip03,Stip04}). We used Proposition~\ref{oldres} 
in \cite{Saha} to identify a configuration, such that every 
contact manifold $(Y,\xi)$ with a contact surgery presentation admitting
this configuration has vanishing contact 
invariant $\chat(\xi)$ (see \cite[Theorem 7.4]{Saha}). In light of this result, 
finding pairs $(Y,K)$ for which the knot Floer homology is zero will allow us to identify additional
configurations in contact surgery diagrams that force the contact element to vanish. Furthermore, 
we demonstrate in \S\ref{computefoone} that results 
like Proposition~\ref{oldres} can be of significant help in calculations when combined
with surgery exact triangles.
Having a closer look at the example in Proposition~\ref{oldres} we see that this is a very
special situation. So, it is natural to seek for additional examples among 
$\sone$-bundles over compact, orientable surfaces. As an application of the
adjunction inequalities, in combination with the symmetries we prove that the search
will be unsuccessful.
\begin{theorem}\label{mythm:01}
Let $Y$ be a $\sone$-bundle over a closed, oriented surface $\Sigma$ and let 
$K^*$ be a knot isotopic to a fiber. Then $\hfkhat(Y,K^*)$ is non-zero if and only
if $Y$ is not $\stwo\times\sone$.
\end{theorem}
We can interpret this result in such a way that knot Floer homology of 
a fiber distinguishes 
the manifold $\stwo\times\sone$ from other $\sone$-bundles.\vspace{0.3cm}\\
We prove Theorem~\ref{mythm:01} in two steps: In the first step we 
reprove Proposition~\ref{oldres} and give explicit calculations for 
genus-$0$ base and non-zero Euler number. In the second step, we prove 
the general result, i.e.~for non-zero genus and arbitrary Euler number. The 
proof will mainly rely on applying knot conjugation symmetry and a suitable 
surgery exact triangle. 
Furthermore, we provide an explicit calculation in \S\ref{computefoone} of the 
case of genus-$1$ base and Euler number $0$. This serves as a demonstration in 
what way results like Proposition~\ref{oldres} can help to do explicit 
calculations.\vspace{0.3cm}\\

\section{Preliminaries}\label{preliminaries}
\subsection{Heegaard Floer homologies}\label{prelim:01:1}
A $3$-manifold $Y$ can be described by a Heegaard diagram, which is a
triple $(\Sigma,\afat,\bfat)$, where $\Sigma$ is an oriented genus-$g$ surface
and $\afat=\{\alpha_1,\dots,\alpha_g\}$, 
$\bfat=\{\beta_1,\dots,\beta_g\}$ are two sets of pairwise disjoint simple closed 
curves in $\Sigma$ called {\bf attaching circles}. 
Each set of curves $\afat$ and $\bfat$ is required to consist of 
linearly independent curves in $H_1(\Sigma,\Z)$. In the 
following we will talk about the curves in the set $\afat$ (resp.~$\bfat$) as  
{\bf $\afat$-curves} (resp.~{\bf $\bfat$-curves}). Without loss 
of generality we may assume that the 
$\afat$-curves and $\bfat$-curves intersect 
transversely. To a Heegaard diagram we may associate the triple
$(\symg,\talpha,\tbeta)$ consisting of the $g$-fold symmetric power of
$\Sigma$, 
\[
  \symg=\Sigma^{\times g}/S_g,
\] 
and the submanifolds $\talpha=\alpha_1\times\dots\times\alpha_g$
and $\tbeta=\beta_1\times\dots\times\beta_g$. We define 
$\cfhat(\Sigma,\afat,\bfat)$ as the free $\ztwo$-module 
generated by the set
$\talpha\cap\tbeta$. In the following
we will just write $\cfhat$. For two intersection
points $\xfat,\yfat\in\talpha\cap\tbeta$ define $\pitwo(\xfat,\yfat)$ to be the set of
homology classes of {\bf Whitney discs} 
$\phi\co\disc\lra\symg$ ($\disc\subset\C$) that 
{\bf connect $\xfat$ with $\yfat$}. The map $\phi$ is called {\bf Whitney} if 
$\phi(\disc\cap\{Re<0\})\subset\talpha$ and $\phi(\disc\cap\{Re>0\})\subset\tbeta$. 
We call $\disc\cap\{Re<0\}$ the {\bf $\afat$-boundary of $\phi$} and
$\disc\cap\{Re>0\}$ the {\bf $\bfat$-boundary of $\phi$}. Such 
a Whitney disc {\bf connects $\xfat$ with $\yfat$} if $\phi(i)=\xfat$ and $\phi(-i)=\yfat$. 
Note that $\pitwo(\xfat,\yfat)$ can be interpreted as the subgroup of elements in
$H_2(\symg,\talpha\cup\tbeta)$ represented by discs with appropriate 
boundary conditions. We endow 
$\symg$ with a symplectic structure~$\omega$. By choosing a path of almost complex 
structures $\mJ_s$ on $\symg$ suitably (cf.~\cite{OsZa01})
all moduli spaces of holomorphic Whitney discs are Gromov-compact manifolds.
Denote by $\modphi$ the set of holomorphic Whitney discs in the equivalence
class $\phi$, and $\mu(\phi)$ the formal dimension of $\modphi$. Denote by 
$\modhatphi=\modphi/\R$ the quotient under the translation action of 
$\R$ (cf.~\cite{OsZa01}). Define $H(x,y,k)$ to be the subset of classes in
$\pitwo(\xfat,\yfat)$ that admit moduli spaces of dimension $k$. Fix a point 
$z\in\Sigma\backslash(\afat\cup\bfat)$ and define the map 
\[
  n_z\co\pitwo(\xfat,\yfat)\lra\Z,\,\phi\lmt\#(\phi,\{z\}\times\symgmo).
\] 
A boundary operator $\parhat\co\cfhat\lra\cfhat$ is given by defining it
on the generators $\xfat$ of $\cfhat$ by
\[
  \parhat\xfat
  =
  \sum_{\yfat\in\talpha\cap\tbeta}
  \sum_{\phi\in H(\xfat,\yfat,1)}
  \!\!\!\!\#\modhatphi\cdot U^{n_z(\phi)}\yfat.
\]
These homology groups are topological invariants of the manifold $Y$. 
We would like to note that not all Heegaard diagrams are suitable
for defining Heegaard Floer homology; there is an additional
condition that has to be imposed called {\bf weak admissibility}
(see~\cite[Definition~4.10]{OsZa01}).

\subsection{Knot Floer Homology}\label{knotfloerhomology}
Given a knot $K\subset Y$, we can specify a certain subclass of 
Heegaard diagrams.
\begin{definition} \label{knotdiagram} A Heegaard 
diagram $(\Sigma,\afat,\bfat)$ is said to
be {\bf adapted} to the knot $K$ if $K$ is isotopic to a knot lying
in $\Sigma$ and $K$ intersects $\beta_1$ once transversely and is
disjoint from the other $\bfat$-circles.
\end{definition}
Since $K$ intersects $\beta_1$ once and is disjoint from the other 
$\bfat$-curves we know that $K$
intersects the core disc of the $2$-handle represented by $\beta_1$ once
and is disjoint from the others (after possibly isotoping the knot $K$).
Every pair $(Y,K)$ admits a Heegaard diagram adapted to $K$.
Having fixed such a Heegaard diagram $(\Sigma,\afat,\bfat)$ we can encode 
the knot $K$ in a pair of points. After isotoping $K$ onto $\Sigma$, 
we fix a small interval $I$ in $K$ containing the intersection point 
$K\cap\beta_1$. This interval should be chosen small enough such 
that $I$ does not contain any other intersections of $K$ with other 
attaching curves. The boundary $\partial I$ of $I$ determines two 
points in $\Sigma$ that lie in the complement of the attaching circles, 
i.e.~$\partial I=z-w$, where the orientation of $I$ is given by the 
knot orientation. This leads to a doubly-pointed Heegaard diagram 
$(\Sigma,\afat,\bfat,w,z)$. Conversely, a doubly-pointed Heegaard 
diagram uniquely determines a topological knot class: Connect 
$w$ with $z$ in the complement of the attaching circles $\afat$ 
and $\bfat\backslash\beta_1$ with an arc $\delta$ that crosses 
$\beta_1$, once. Connect $z$ with $w$ in the complement of $\bfat$
using an arc $\gamma$. The union $\delta\cup\gamma$ represents the 
knot class $K$. The orientation of $K$ is given by orienting $\delta$ such 
that $\partial\delta=z-w$. \vspace{0.3cm}\\
The knot chain complex $\cfkom(Y,K)$ is the free $\ztwo[U]$-module generated
by the intersections $\talpha\cap\tbeta$. Analogous as above we define
$\mMhat^{(i,j)}_{(\afat,\bfat)}(\xfat,\yfat)$ as the holomorphic Whitney disks
connecting $\xfat$ with $\yfat$ such that $(n_z(\phi),n_w(\phi))$ equals 
$(i,j)$, after modding out the translation action. The boundary operator $\partial^{\bullet,-}_{\afat\bfat}$, 
for $\xfat\in\talpha\cap\tbeta$, is defined by
\[ 
 \partial^{\bullet,-}_{\afat\bfat}(\xfat) 
 =
 \sum_{\yfat\in\talpha\cap\tbeta,j\geq0} 
 \!\!\!\!\#\Bigl(\mMhat^{(0,j)}_{(\afat,\bfat)}(\xfat,\yfat)\Bigr)
 \cdot U^{j}\yfat.
\]
The associated homology theory is denoted by $\hfkom(Y,K)$. By
setting $U=0$ we obtain the theory $\hfkoo(Y,K)$ which we also denote 
by $\hfkhat(Y,K)$. It is also possible to define variants such 
as $\hfkop$ and $\hfkoi$. For details we point the reader to \cite{OsZa03}.
Concerning admissibility, note that for the versions $\hfkoc$ we restrict
to doubly pointed Heegaard diagrams $(\Sigma,\afat,\bfat,w,z)$ such that
the single pointed diagram $(\Sigma,\afat,\bfat,z)$ is weakly 
admissible (see~\cite[Definition~4.10]{OsZa01}).
Finally, to justify our notation, observe, that it is possible to {\it swap the roles}
of $z$ and $w$ and so derive knot Floer theories denoted by $\hfkmo$, $\hfkio$ and 
$\mbox{\rm HFK}^{+,\bullet}$.\vspace{0.3cm}\\
For a treatment of cobordism maps between the various versions of knot Floer
homology we point the reader to \cite[\S8 and \S9]{Saha04}.

\section{A new adjunction inequality}
Ozsv\'ath and Szab\'o derived adjunction inequalities for the ordinary Heegaard 
Floer homology in \cite{OsZa02} and gave a knot theoretic version using Seifert 
surfaces in \cite{OsZa04}. An adjunction inequality usually gives an upper bound on 
the quantity $\bigl|\bigl<c_1(\fraks),[F]\bigr>\bigr|$ for arbitrary embedded 
surfaces $F$ and $\spinc$-structures $\fraks$ for which the Floer homology groups
are non-zero. Suppose we are given a closed, oriented $3$-manifold $Y$ and in 
it a homologically non-trivial, closed surface $F$ of genus $g(F)$. The main 
observation to prove our result is to see that we can find a Heegaard diagram 
which is adapted to the surface.
\begin{lem}[Lemma~7.3.~of~\cite{OsZa02}]\label{oszahelp} 
Suppose $F\subset Y$ is a homologically 
non-trivial, embedded two-manifold with $g(F)>0$, then 
$Y$ admits a genus $g$ Heegaard diagram
$(\Sigma,\afat,\bfat)$, with $g>2g(F)$, containing a periodic domain $\mathcal{P}$ 
representing $[F]$, all of whose multiplicities are one or zero. Moreover, 
$\mathcal{P}$ is a connected surface whose Euler characteristic is equal to 
$-2g(F)$, and $\mathcal{P}$ is bounded by $\beta_1$ and $\alpha_{2g+1}$.
\end{lem}
In the following, we will call such a Heegaard diagram {\bf $F$-adapted}. With this 
Heegaard diagram Ozsv\'ath and Szab\'o were able to derive the following chern class 
formula.
\begin{lem}[Proposition~7.4.~of~\cite{OsZa02}]\label{chernclassformula} If $\xfat=\{x_1,\dots,x_g\}$ is an intersection point, and $z$ is chosen in 
the complement of the periodic domain $\mathcal{P}$ of Lemma~\ref{oszahelp}, then
\[
 \bigl<c_1(\fraks_z(\xfat)),\mathcal{H}(\mathcal{P})\bigr>
 =2-2g+2\#\{x_i\;\mbox{\rm in the interior of } \mathcal{P}\}.
\]
\end{lem}
With this formula at hand we see that whenever there is an intersection point 
$\xfat$ whose associated $\spinc$-structure $\fraks_z(\xfat)$ equals $\fraks$ we 
have that
\[
 -\bigl<c_1(\fraks),[F]\bigr>\leq 2g(F)-2
\]
For our purposes we have to see that $F$-adaptedness and adaptedness to a 
knot $K$ can be achieved simultaneously.
\begin{prop}\label{mydiagram} Suppose we are given an embedded 
surface $F\subset Y$ with $g(F)>0$ which is
homologically non-trivial and, further, suppose we are given a knot $K\subset Y$ such 
that $\#(K\cap F)\leq 1$. Then, there is a $F$-adapted Heegaard diagram (in the 
sense of Lemma~\ref{oszahelp}) which is adapted to the knot $K$.
\end{prop}
\begin{proof} Choose a tubular neighborhood 
$\nu F=F\times[-1,1]$ of $F$, then the manifold 
$Y\backslash\nu F$ admits a handle decomposition $\dom_2$ relative to 
$\partial\bigl(Y\backslash\nu F\bigr)$ by $1$-handles
 $\sh^{(3,1)}_1,\dots,\sh^{(3,1)}_l$, 
$2$-handles $\sh^{(3,2)}_1,\dots\sh^{(3,2)}_k$ and one single $3$-handle 
$\sh^{(3,3)}$ 
(cf.~\cite[p.~104]{GoSt}). The surface $F$ admits a handle decomposition into
a single $0$-handle $\sh^{(2,0)}$, $1$-handles 
$\sh^{(2,1)}_1,\dots,\sh^{(2,1)}_{2g(F)}$ and a single $2$-handle $\sh^{(2,2)}$. 
Crossing a $2$-dimensional $k$-handle $\sh^{(2,k)}$ with the interval 
$[-1,1]$ it transforms into a $3$-dimensional $k$-handle $\shb^{(3,k)}$. Hence, 
we obtain a handle decomposition $\dom_1$ of the tubular neighborhood $\nu F$. 
Observe, that the $2$-handle $\shb^{(3,2)}$ comes from $h^{(2,2)}$ which was used 
to cap-off the boundary of
\begin{equation}
 \sh^{(2,0)}
 \cupb
 \sh^{(2,1)}_1
 \cupb
 \dots
 \cupb
 \sh^{(2,1)}_{2g(F)}.
 \label{surfhand}
\end{equation}
By isotopies of the attaching spheres of the $1$-handles we may think the handle 
$\shb^{(3,2)}$ to be attached after the handles 
$\sh^{(3,1)}_1,\dots,\sh^{(3,1)}_l$. Thus, the pair $(\dom_1,\dom_2)$ induces a 
Heegaard decomposition, where the handlebody $H_0$ is given by the union of 
$\shb^{(3,0)}$ with the handles 
\[
\shb^{(3,1)},
\dots,
\shb^{(3,1)}_{2g(F)},
\sh^{(3,1)}_1,\dots,\sh^{(3,1)}_l.
\]
Denote by $\bfat_1$ the attaching sphere of $\shb^{(3,2)}$. The curve $\beta_1$ 
bounds a surface $S$, diffeomorphic to 
$F\backslash D^2$, in the Heegaard surface $\Sigma$, since the surface $S$ admits a handle decomposition given in 
(\ref{surfhand}). The surface $F$ cannot separate and, thus, there has to be a $1$-handle, $\sh^{(3,1)}_{2g+1}$-say, that connects $S$ with 
$\Sigma\backslash D$. Using 
isotopies of the attaching spheres and handle slides of the $1$-handles we again 
may assume that this is the only $1$-handle with this property. The attaching 
sphere, $\afat_{2g+1}$, and $\bfat_1$ bound a surface whose associated homology element, 
given by capping it off with the core discs given by $\afat_{2g+1}$ and $\bfat_1$, 
equals $[F]$. The associated Heegaard diagram is $F$-adapted. 
To bring in the knot $K$ we have to 
cover two cases. If $K$ and $F$ are disjoint, we may replace $\dom_1$ by $\dom'_1$ 
which is a handle decomposition of $\nu F$ 
extended to $\nu K$ by an additional $0$-handle and $1$-handle. Denote by $\dom_2$ a handle decomposition of $Y\backslash\bigl(\nu F\cup\nu K\bigr)$ relative to 
$\partial\bigl(\nu F\cup\nu K\bigr)$. Observe, that since
$\nu F\cup\nu K$ is disconnected there has to be a $1$-handle connecting these.
Without loss of generality we may think this $1$-handle to connect the $0$-handles 
in $\dom'_1$. Thus, the pair 
$(\dom'_1,\dom_2)$ again induces a Heegaard decomposition by the same reasoning as 
above. The resulting Heegaard diagram is $F$-adapted and $K$-adapted. 
In case $K$ and $F$ intersect in a unique point we have to be a bit more 
careful: Denote by $\dom_1$ a handle decomposition of $\nu F$, with handles denoted 
like above, such that $\shb^{(3,0)}$ equals $\nu F\cap\nu K$. With this arrangement 
we can extend $\dom_1$ to a decomposition $\dom'_1$ of $\nu K\cup \nu F$ by adding 
a $1$-handle
$\sh^{(3,1)}_*$ given by $\nu K\backslash\shb^{(3,0)}$. Denote by $\dom_2$ a handle 
decomposition of $Y\backslash\bigl(\nu K\cup\nu F\bigr)$ relative to 
$\partial\bigl(\nu K\cup\nu F)$ with the handles denoted the same way as above. 
The $2$-handle $\shb^{(3,2)}$ may be thought of as being attached after the 
$1$-handles in the decomposition $\dom_2$. So, the pair $(\dom'_1,\dom_2)$ 
induces a Heegaard decomposition. We rename the handle $\sh^{(3,1)}_*$
to $\sh^{(3,1)}_{2g+1}$. The resulting Heegaard diagram is $F$-adapted and $K$-adapted.
\end{proof}

Before we prove 
Theorem~\ref{myadjuncineq} we have to recall how to recover an oriented knot from 
a doubly-pointed Heegaard diagram $(\Sigma,\afat,\bfat,w,z)$. Here, it is 
opportune to look into the Morse-theoretic picture. We obtain a Heegaard diagram 
from a self-indexing Morse-function $f\co Y\lra\R$. If the associated Heegaard 
decomposition is $K$-adapted, then $K$ is isotopic to the union of two flow lines 
connecting the index-$0$ critical point with the index-$3$ critical point. The 
Heegaard surface $\Sigma=f^{-1}(3/2)$ is oriented such that every flow line 
intersects $\Sigma$ positively. The two flow lines determining $K$ intersect 
$\Sigma$ in two points $z$ and $w$. Denote by $\gamma_z$ and $\gamma_w$ the 
respective flow lines, then $K$ is isotopic to $\gamma_z-\gamma_w$. Thus, the knot 
$K$ when isotoped into $\Sigma$ runs from $z$ to $w$ in the complement of the 
$\bfat$-curves and from $w$ to $z$ in the complement of the $\afat$-curves.
\begin{proof}[Proof of Theorem~\ref{myadjuncineq}] Suppose we have chosen 
an $F$-adapted and $K$-adapted Heegaard diagram
$(\Sigma,\afat,\bfat)$ with base points $w$ and $z$.  Given that the knot
Floer homology $\hfkoc(Y,K;\fraks)$ is non-zero, there has to be an 
intersection point $\xfat\in\talpha\cap\tbeta$ such that $\fraks_z(\xfat)=\fraks$.
From the proof of Proposition~\ref{mydiagram} we know that the handle determing the 
tubular neighborhood of $K$ bounds the periodic domain $\mathcal{P}$ 
(cf.~Lemma~\ref{oszahelp}) in case $\#(K\cap F)=1$ and they are disjoint otherwise. 
Thus, either both $w$ and $z$ lie outside of $\mathcal{P}$ or, depending on the 
orientation of $K$, one of them lies inside. Observe, that both lie outside of $\mathcal{P}$ in case $F$ and $K$ are disjoint. This means, that there is also a $(-F)$-adapted and $K$-adapted Heegaard diagram with both $w$ and $z$ lying outside of the associated periodic domain. Thus, in this case, we may apply Lemma~\ref{chernclassformula} to both get
\begin{eqnarray*}
2-2g(F)&\leq&\bigl<c_1(\fraks),[F]\bigr>\\
 2-2g(F)&\leq&\bigl<c_1(\fraks),[-F]\bigr>=-\bigl<c_1(\fraks),[F]\bigr>
\end{eqnarray*}
and, consequently,
\[
 \bigl|\bigl<c_1(\fraks),[F]\bigr>\bigr|\leq 2g(F)-2.
\]
In case $K$ and $F$ intersect, either $z$ or $w$ lie inside of $\mathcal{P}$. 
To use the 
Morse theoretic picture, recall that $K$ is isotopic to the knot 
$\gamma_z-\gamma_w$ where $\gamma_z$ and $\gamma_w$ are the gradient flow lines 
through $z$ and $w$ of the Morse function defining the Heegaard splitting. Thus, 
$K$ intersects the Heegaard surface in $z$, positively, and in $w$, negatively. 
Since $w$ should lie in $\mathcal{P}$ and the orientation of the Heegaard surface 
and of $F$ coincide, we see that $K$ has to intersect $F$, negatively. With this in 
place, we may apply Lemma~\ref{chernclassformula} to get the desired inequality 
$(i)$. 
\end{proof}
\begin{figure}[t]
\labellist\small\hair 2pt
\endlabellist
\centering
\includegraphics[width=3.5cm]{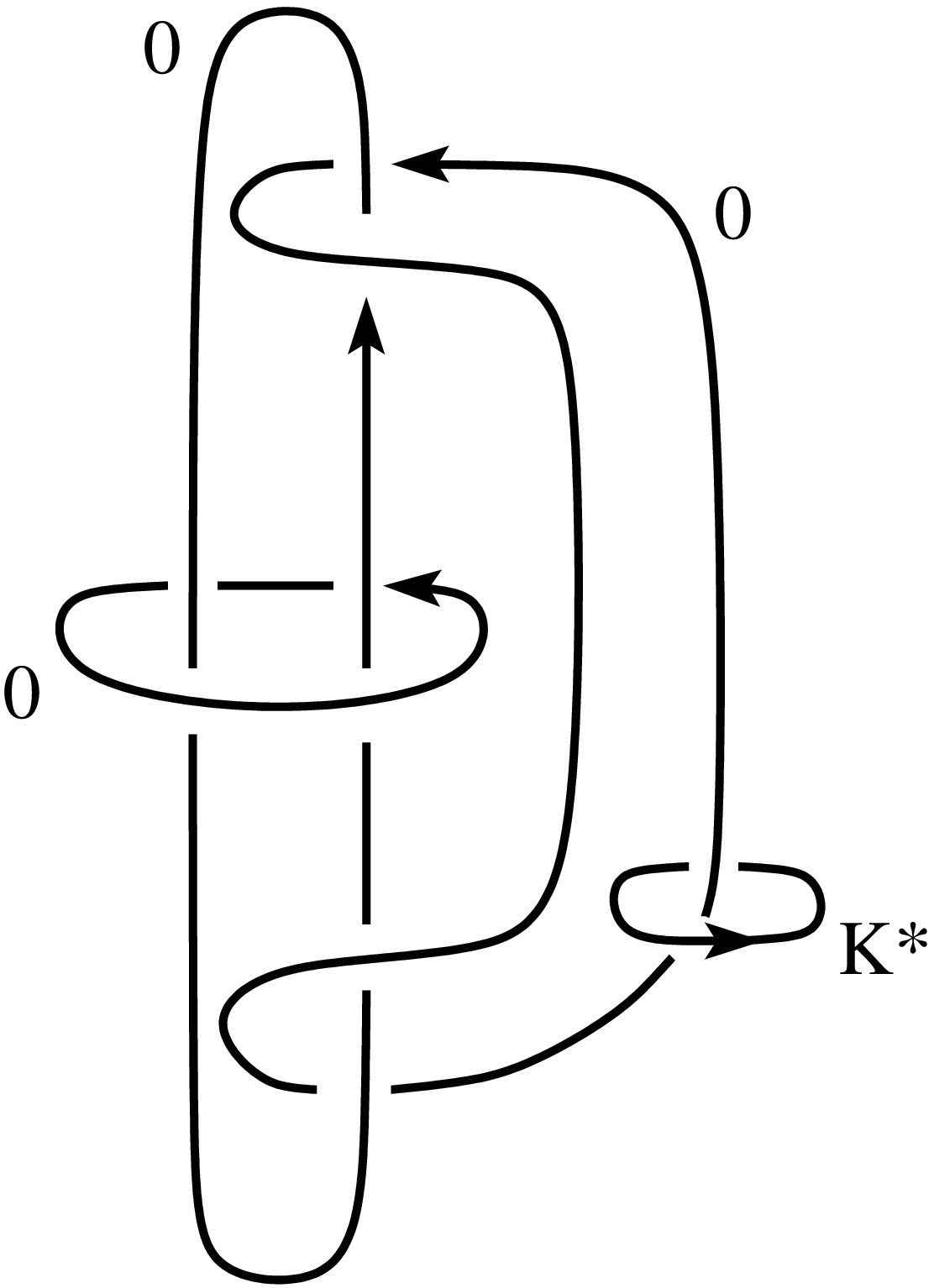}
\caption{The manifold $T^3$ with positively oriented fiber $K^*$.}
\label{Fig:tthreeorient}
\end{figure}
Before delving into the proof of Theorem~\ref{mymapadjunc} we would like 
to point the reader to 
\cite[Proposition~2.1]{Stip02} and its proof.
Lisca and Stipsicz prove an adjunction inequality for cobordism maps 
$\Fhat_{K,\fraks}$
in case $b_1(Y)=0$. Given that $[F]^2=n$ they blow-up the cobordism $W$ at $n$ 
distinct points in the complement of $F$. Denote by $\What$ the new cobordism. 
Lisca and Stipsicz split up the cobordism into 
$\What_1$ and $\What_2$ where $\What=\What_1\cup_\partial\What_2$. By the 
composition law of cobordism maps, a homological computation and the 
adjunction inequalities for the Heegaard Floer groups they show that the 
map $\Fhat_{\What}$ factorizes through a trivial Heegaard Floer group. For 
the homological computation they need the surjectivity of the morphism
\[
 \iota_*
 \co
 H_2(\What;\Z)
 \lra
 H_2(\What,Y;\Z)
\]
which is a consequence of $b_1(Y)=0$. However, by the Mayer-Vietoris sequence
\[
 \xymatrix@C=2pc@R=0.1pc
 {
 \dots\ar[r]
 &
 H_1(S^1)\ar[r]^{\,\hspace{-1.5cm}\iota_1^*\oplus\iota_2^*}
 & 
 H_1(Y\times[0,1])
 \oplus
 H_1(\sh^{(4,2)})\ar[r]
 &
 H_1(W)\ar[r]
 &
 0\\
 &
 [S^1]\ar@{|->}[r]
 & 
 \,\hspace{0.2cm}[K]\oplus 0\\
 }
\]
we see that surjectivity of $\iota_*$ already follows by assuming 
that $K$ is homologically trivial by comparing with the long exact sequence
of the pair $(W,Y)$. Thus, we may follow the lines of their 
proof but impose the relaxed condition that $[K]=0$ instead 
of $b_1(Y)=0$. Furthermore, in the proof of \cite[Proposition 2.1]{Stip02} 
Lisca and Stipsicz use three properties of cobordism maps: The blow-up 
formula (see~\cite[Theorem~3.7]{OsZa03}), conjugation invariance 
(see~\cite[Theorem~3.6]{OsZa03}) and the composition 
law (see~\cite[Theorem~3.4]{OsZa03}). Our proof
of the adjunction inequality will consist of showing that for cobordism maps in 
knot Floer homologies there exists a blow-up formula and a composition law and 
finally, to see that we can split the knot 
cobordism $\What$ (see~\cite[Definition~8.1]{Saha04}) into two knot 
cobordisms $\What_1$ and $\What_2$.
\begin{proof}[Proof of Theorem~\ref{mymapadjunc}] The conjugation invariance is 
given as part of Proposition~\ref{knotconjug}.\vspace{0.3cm}\\
{\bf Composition law.} The composition law essentially requires the 
associativity of the pair of pants pairing $\foc_{\afat,\bfat\gfat}$ 
defined by counting holomorphic triangles (cf.~\cite[\S8.3]{Saha04}). The 
necessary statement can be easily 
derived from the corresponding associativity statement about 
$\fhat_{\afat,\bfat\gfat}$ by 
applying \cite[Theorem~1.2]{Saha04} and \cite[Theorem~1.3]{Saha04} with 
the techniques provided in that paper (see~the proof of 
\cite[Theorem~8.7]{Saha04} or \cite[Example~4.1]{Saha04}).\vspace{0.3cm}\\
{\bf Blow-up formula.} The proof Ozsv\'{a}th and Szab\'{o} give carries 
over verbatim to our case with the following assumption imposed: given 
a knot cobordism $W=([0,1]\times Y,\phi)$ (cf.~\cite[Definition~8.1]{Saha04}) we 
blow-up in the complement of the image of $\phi$. We see that for each $\spinc$-structure $\frakt$ over
$Y$ the corresponding map
\[
 \Foc_{W;\fraks}
 \co
 \hfkoc(Y,\frakt)
 \lra
 \hfkoc(Y,\frakt)
\]
is the identity for $\left.\fraks\right|_{\{0\}\times Y}=\frakt$
with $\left<c_1(\fraks),E\right>=\pm1$ and it 
vanishes, otherwise.\vspace{0.3cm}\\
{\bf Splitting up knot cobordisms.} We blow up the cobordism $W$ at $n$ distinct
points in the complement of the surface $F$ and denote the resulting cobordism by
$\What$. We take the surface $F$ and pipe them to $n$ exceptional spheres as done
in Lisca and Stipsicz's work to obtain a surface $\Fhat$. We, further, choose a 
properly embedded arc $\gamma$ connecting $Y$ with $F$ which is disjoint from $Y$ 
and $F$ away from its endpoints. A regular neighborhood 
of $Y\cup\gamma\cup\Sigma$ is
denoted by $\What_1$ and the closure of the complement denoted by $\What_2$. So 
far, this is the construction of Lisca and Stipsicz, used in their proof 
of \cite[Proposition~2.1]{Stip02}. In our case, we make a special choice 
for $\gamma$: If $[0,1]\times K$ is disjoint from $F$, we choose $\gamma$ 
to be disjoint from $[0,1]\times K$. If $[0,1]\times K$ and $F$ 
intersect once, at the point $(t,s)$ say, we set $\gamma=[0,t]\times\{s\}$.
Then, for an appropriate regular neighborhood of $Y\cup\gamma\cup F$ the
intersection of $[0,1]\times K$ with $\partial\What_1$ will be a homologically 
essential embedded knot in $[0,1]\times K$. Hence, the embedded cylinder 
$[0,1]\times K$ splits into two cylinders, one embedded in $\What_1$ and 
one embedded in $\What_2$. We have that $\partial\What_1=Y\#(\sone\times\Fhat)$ 
and we obtain a knot $K'$ in this manifold that intersects $\Fhat$ once, 
negatively.\vspace{0.3cm}\\
With these remarks done, we follow the lines of the proof 
of \cite[Proposition 2.1]{Stip02}.
\end{proof}

\section{Symmetries}\label{symmetries}
Symmetries for knot Floer homologies with $[K]=0$ were discussed by Ozsv\'{a}th and
Szab\'{o} in \cite{OsZa04}. Here, we provide a discussion, which especially covers
the case $[K]\not=0$.
\begin{proof}[Proof of Proposition~\ref{mysymmetry}] Given a pair $(Y,K)$ we choose a doubly-pointed Heegaard diagram
$\mh_1=(\Sigma,\afat,\bfat,w,z)$ which is adapted to $K$ and which is 
extremely-weakly admissible (see \cite{Saha}). By swapping the points $z$ and 
$w$ we reverse the orientation of the associated knot, i.e.~the Heegaard 
diagram $\mh_2=(\Sigma,\afat,\bfat,z,w)$ represents the pair $(Y,-K)$.
The chain groups $\cfkhat(\mh_1)$ and $\cfkhat(\mh_2)$ are canonically isomorphic
as both are generated by the same set, i.e.~both are generated by $\talpha\cap\tbeta$.
Let 
\[
 \pointswap
 \co
 \cfhat(\mh_1)
 \lra
 \cfhat(\mh_2)
\]
be this canonical isomorphism. Denote by $\parhat^{\mh_i}$, $i=1,2$ the knot Floer
differential associated to the diagrams $\mh_i$. Both differentials count holomorphic 
discs $\phi$ with boundary conditions in $\talpha$ and $\tbeta$ such that $n_z(\phi)=n_w(\phi)=0$.
Consequently, the equality
\[
  \parhat^{\mh_1}=\parhat^{\mh_2}
\]
holds. Hence, $\pointswap$ induces an isomorphism between the associated knot Floer homologies
we will also denote by
\[
 \pointswap
 \co
 \hfkhat(Y,K;\fraks)
 \overset{\cong}{\lra}
 \hfkhat(Y,-K;\frakt_\fraks)
\]
where $\fraks,\frakt_\fraks\in\spinc(Y)$ are suitable pairs of $\spinc$-structures. To be precise,
given an intersection point 
$\xfat\in\talpha\cap\tbeta$ with $\fraks=\fraks_z(\xfat)$, then
$\frakt_\fraks$ is the $\spinc$-structure associated to $\pointswap(\xfat)$. Recall, that $\pointswap(\xfat)=\xfat$,
but now $\xfat$ has to be interpreted as a generator of $\cfkhat(\mh_2)$. Since we obtained $\mh_2$
from $\mh_1$ by swapping the points $z$ and $w$ we have that
\[
 \frakt_\fraks=\fraks_w(\xfat).
\]
Applying the formula given in \cite[Lemma 2.19]{OsZa01} (or looking at the end of its proof) we see that
\[
 \fraks_z(\xfat)
 -
 \fraks_w(\xfat)
 =-PD[\gamma_z-\gamma_w].
\]
By definition, $[K]=[\gamma_z-\gamma_w]$ and, hence, $\frakt_\fraks=\fraks+PD[K]$.
\end{proof}
\noindent On the $\spinc$-structures of a manifold $Y$ there is an operation 
called {\bf conjugation} given by the following algorithm: Let $[X]$ be a 
homology class of unit-length vector fields on $Y$ (i.e.~a $\spinc$-structure) and 
let $X$ be a representative. Then we define a map 
\[
 \mathcal{J}
 \co
 \spinc(Y)\lra\spinc(Y)
\]
by sending $[X]$ to $[-X]$. The knot Floer homologies fulfill a conjugation 
symmetry. 
\begin{prop}[see~Proposition~3.9~of~\cite{OsZa04}]\label{conjuginvar} Let $Y$ be a closed, oriented $3$-manifold and $K$ a knot. Then
we have that
\[
 \hfkoc(Y,K;\fraks)
 \cong
 \hfkoc(Y,-K;\mJ(\fraks)).
\]
We denote by $\mJ_Y$ the associated isomorphism.
\end{prop}
\begin{proof} The proof goes exactly as in the case of homologically trivial knots
given by Ozsv\'{a}th and Szab\'{o}.
\end{proof}
\begin{proof}[Proof of Corollary~\ref{mysymmetry}]
The isomorphism is given by
\[
  \mN_{(Y,K);\fraks}=\mJ_{Y;\fraks+PD[K]}\circ\pointswap_\fraks
\] or equivalently by $\pointswap_\fraks\circ\mJ_{Y;\fraks}$.
\end{proof} 
We call the isomorphism $\mN_{(Y,K);\fraks}$ the {\bf knot conjugation isomorphism}.
These morphism allow us to prove that maps induced by cobordisms are
symmetric with respect to this knot conjugation.
\begin{proof}[Proof of Proposition~\ref{knotconjug}] We will prove this 
for cobordisms coming from $2$-handle attachments. The discussion for 
$1$-handles and $3$-handles can done without difficulty. 
Suppose $W$ is given by attaching a $2$-handle along a knot $K'$ in 
the complement of $K$. We define the map $\Foo_{W,\fraks}$ 
(see~\cite{Saha04}) in the following way: Let $(\Sigma,\afat,\bfat,w,z)$ be 
a Heegaard diagram which is $K$-adapted and $K'$-adapted. The $2$-handle 
attachment along $K'$ corresponds to a surgery along $K'$. Performing 
this surgery, we obtain a third set of attaching circles $\gfat$. The
doubly-pointed Heegaard triple diagram $(\Sigma,\afat,\bfat,\gfat,w,z)$
determines the cobordism $W$. The map $\cfoo_{\afat,\bfat\gfat}$ is 
defined by counting holomorphic triangles in this diagram and the induced
cobordism map (in homology) is denoted by $\Foo_{W}$. 
Knot conjugation is given as a composition of ordinary conjugation 
$\mJ$ and the point-swap isomorphism $\pointswap$. Both isomorphisms, on 
the chain level, are the identity. We have to compare the following two 
maps
\[
 \begin{array}{rccc}
  \cfoo_{\afat,\bfat\gfat}(\,\cdot\,,\zfat^+_{\bfat\gfat})
  \co&
  \cfkhat(\Sigma,\afat,\bfat,w,z)&
  \lra&
  \cfkhat(\Sigma,\afat,\gfat,w,z)\\
  \cfoo_{\bfat\gfat,\afat}(\,\cdot\,,\zfat^+_{\gfat\bfat})\co&
  \cfkhat(-\Sigma,\bfat,\afat,z,w)&
  \lra&
  \cfkhat(-\Sigma,\gfat,\afat,z,w)
 \end{array}
\]
we will first describe in the following: The Heegaard diagram 
$(\Sigma,\bfat,\gfat,w,z)$ is an adapted Heegaard diagram for the 
pair $(\sthree\#^{g-1}(\stwo\times\sone),U)$ where $U$ is the unknot. There 
is a subgroup of the associated Heegaard Floer groups we can call the 
top-dimensional homology (cf.~\cite{OsZa01} and \cite{OsZa02}). Let 
$\hattheta^+_{\bfat\gfat}$ be one of its generators. Recall that there is 
an intersection point 
$\zfat_{\bfat\gfat}^+\in\tbeta\cap\tgamma$ such that 
$[\zfat_{\bfat\gfat}^+]$ is this generator. Suppose we are 
given $\xfat\in\talpha\cap\tbeta$, $\yfat\in\talpha\cap\tgamma$ and 
a path $\mJ_s$ of admissible almost complex 
structures. Denote by $\mM^{\mJ_s}_{(\afat,\bfat,\gfat)}(\xfat,\zfat,\yfat)$ 
the moduli space of $\mJ_s$-holomorphic Whitney triangles connecting 
$\xfat$ with $\yfat$ through $\zfat$ with boundary conditions in 
$\talpha$, $\tbeta$ and $\tgamma$ with Maslox index $0$. The first map 
is defined via
\[
 \left.\cfoo_{\afat,\bfat\gfat}(\xfat,\zfat^+_{\bfat\gfat})\right|_\yfat
 =
 \#\mM^{\mJ_s}_{(\afat,\bfat,\gfat)}(\xfat,\zfat^+_{\bfat\gfat},\yfat).
\]
In a similar fashion the second map is defined. Observe, that using 
conjugation we have an identification
\[
  \mM^{\mJ_s}_{(\afat,\bfat,\gfat)}(\xfat,\zfat^+_{\bfat\gfat},\yfat)
  \cong
  \mM^{-\mJ_s}_{(\gfat,\bfat,\afat)}(\xfat,\zfat^+_{\gfat\bfat},\yfat).
\]
The intersection point $\zfat^+_{\beta\gamma}$ can be 
interpreted as sitting in $(-\Sigma,\gfat,\bfat)$ and, due to the 
orientation change of the surface $\Sigma$, the point $\zfat^+_{\beta\gamma}$ 
is a representative of $\hattheta^+_{\gfat\bfat}$. Thus, in the new diagram 
we can interpret the point $\zfat^+_{\bfat\gfat}$ as $\zfat^+_{\gfat\bfat}$. 
Hence, we have
\begin{eqnarray*}
  \left.\cfoo_{\afat,\bfat\gfat}(\xfat,\zfat^+_{\bfat\gfat})\right|_\yfat
  &=&
  \#\mM^{\mJ_s}_{(\afat,\bfat,\gfat)}(\xfat,\zfat^+_{\bfat\gfat},\yfat)
  =
  \#\mM^{-\mJ_s}_{(\gfat,\bfat,\afat)}(\xfat,\zfat^+_{\gfat\bfat},\yfat)\\
  &=&
  \left.\cfoo_{\bfat\gfat,\afat}(\xfat,\zfat^+_{\gfat\bfat})\right|_\yfat.
\end{eqnarray*}
We have shown that on homology we have
\begin{equation}
 \Foo_{W}=\mN_{(Y',K')}\circ\Foo_{W}\circ\mN_{(Y,K)}.
 \label{eq:refine}
\end{equation}
To prove the refined statement given in the proposition, we have to see that
there is a map
\[
 \mN\co\spinc(W)\lra\spinc(W)
\]
that refines the equation $(\ref{eq:refine})$. Observe, that knot conjugation is
a combination of point-swap symmetry and conjugation symmetry, the former swapping
the base points and the latter swapping the roles of $\afat$ and $\bfat$ and altering the surface orientation. From the considerations we have given it is 
easy to derive that
\[
 \Foc_{W,\fraks}
 =
 \mJ_{Y'}\circ\Foc_{W,\mJ(\fraks)}\circ\mJ_{Y}
\]
(cf.~\cite{OsZa03}). To define $\mN$, we first have the see in what way swapping 
base points acts on $\spinc$-structures. Define $\mS_z\subset\spinc(W)$ to be 
the subset of $\spinc$-structures realized by homotopy classes of Whitney 
discs in the Heegaard triple diagram $(\Sigma,\afat,\bfat,\gfat,w,z)$ and, correspondingly, define $\mS_w$ as the subset of 
$\spinc$-structures realized by homotopy classes of Whitney triangles in 
the Heegaard triple diagram $(\Sigma,\afat,\bfat,\gfat,z,w)$. Observe, that 
our considerations already show that swapping base points induces a map
\[
 \eta\co\mS_z\lra\mS_w
\]
such that
\[
 \Foo_{W,\fraks}
 =
 \pointswap_{Y'}\circ\Foo_{W,\eta(\fraks)}\circ\pointswap_{Y}.
\]
However, we would like to see that $\eta$ extends to a map on the set
$\spinc(W)$ by showing that $\fraks-\eta(\fraks)$ does not depend on $\fraks$. 
In fact, it is a shift with a constant class. Given a Withney triangle $\phi$,
Ozsv\'{a}th and Szab\'{o} in \cite[\S 8.1.4]{OsZa01} construct an associated $\spinc$-structure $\fraks_z(\phi)$. Performing their construction to determine
$\fraks_z(\phi)$ and $\fraks_w(\phi)$, we see that these two $\spinc$-structures
will differ only in a tubular neighborhood of $F_0^{z}\cup F_0^{w}$. We basically 
apply Ozsv\'{a}th and Szab\'{o}'s notation, however, we added the superscript-$z$
and $w$ to $F_0$ to indicate the base point used. The
construction of both $F_0^{z}$ and $F_0^w$ does not depend on the Whitney 
triangle $\phi$. Furthermore, the $\spinc$-structures $\fraks_z(\phi)$ will be
independent of $\phi$ in a small tubular 
neighborhood of $F_0^{z}$ and the same holds for $\fraks_w(\phi)$ 
and $F_0^{w}$. Since we are just looking for Whitney triangles
with both $n_z(\phi)=n_w(\phi)=0$, we will also have that $\fraks_z(\phi)$ will 
be independent of $\phi$ at $F_0^w$ and, vice versa, $\fraks_w(\phi)$ will 
be independent of $\phi$ at $F_0^z$. Thus, the difference class
\[
 \fraks_z(\phi)-\fraks_w(\phi)
\]
is a constant multiple of the the Poincar\'{e} dual of the 
class $[F_0^z\cup F_0^w]$ in $H_2(W,\partial W)$. Thus, there
is a homology class $c$ such that $\fraks-\eta(\fraks)=PD[c]$. 
So, $\eta$ can be extended to a map on $\spinc(W)$. Hence, defining
\[
 \mN\co\spinc(W)\lra\spinc(W)
\]
as the composition $\mJ\circ\eta$ we obtain the adequate refinement of equation
$(\ref{eq:refine})$.
\end{proof}
\begin{rem} With more effort it is possible to determine the shifting on
$\spinc$-structures of $W$ induced by point-swap symmetry. However, since
we do not need an explicit calculation for our purposes, we omitted this
step.
\end{rem}
\section{Implications to knot Floer homology}\label{app01}
Before we continue with applications of the derived symmetries, 
we would like to remind the reader of the interpretation of $\spinc$-structures 
as homology classes of vector fields (cf.~\cite{turaev}): A $\spinc$-structure 
$\fraks$ on a $3$-manifold is an equivalence class of unit-length vector fields 
where two vector fields are defined 
to be equivalent if they are homotopic outside of a ball. With this description we 
can characterize a $\spinc$-structure by the homology class of a link $L_\fraks$. 
Furthermore, the first chern class of the $\spinc$-structure and the link are 
related as follows:
\[
 c_1(\fraks)=PD[2\cdot L_\fraks].
\]
Thus, given a surface $F$ inside a $3$-manifold, the quantity 
$\bigl<c_1(\fraks),[F]\bigr>$ equals twice the intersection number of $L_\fraks$ 
with $F$, i.e.~
\begin{equation}
 \bigl<c_1(\fraks),[F]\bigr>
 =2\cdot\#(L_s,F).\label{eq:chernformula}
\end{equation}
With this at hand we are able to prove the statement of Theorem~\ref{mysym:kfh}.
\begin{proof}[Proof of Theorem~\ref{mysym:kfh}] This immediately follows from 
the fact that the knot conjugation
\[
 \mN\co\spinc(Y)\lra\spinc(Y)
\]
is a bijection without fixed points in case $[K]$ cannot be divided by two. 
If $\hfkhat(Y,K;\fraks)$ has odd rank, then the group 
$\hfkhat(Y,K;\mN(\fraks))$ has odd rank, too. Supposing that
 $\fraks=\mN(\fraks)=\mJ(\fraks)-PD[K]$ we derive that 
$PD[K]=2\fraks$ which contradicts the assumptions. So, the refined groups 
come in pairs which both have the same rank.
\end{proof}
Indeed, the result given in Theorem~\ref{mysym:kfh} underpins the difference 
between the case $[K]=0$ and $[K]\not=0$. 
\begin{proof}[Proof of Proposition~\ref{thm:van}] The proof combines the adjunction inequalities given in 
Theorem~\ref{myadjuncineq}, the point-swap symmetry given in 
Proposition~\ref{mysymmetry} and the knot conjugation symmetry given in
Proposition~\ref{knotconjug}.\vspace{0.3cm}\\
First we look at the different conditions posed in $(a)$ and $(b)$. The two
conditions 
$\#(L\cap\Sigma_S)$ is non-zero, odd and $lk(K,L)\not=\#(\Sigma\cap\Sigma_S)$
are equivalent to saying that $\Sigma_S$ and $L$ intersect non-trivial, in an
odd number of points and at least one of these intersections is negative. On the 
other hand the conditions $\#(L\cap\Sigma_S)$ is non-zero, odd and 
$lk(K,L)=\#(\Sigma\cap\Sigma_S)$ means that $L$ and $\Sigma_S$ intersect 
non-trivially, in an odd number of points and all of these intersections are 
positive. The second set of conditions in (a) and the second set of conditions in 
(b) can be covered by discussing $L$ and $\Sigma_S$, intersecting in an even number 
of points. Thus, there are basically three cases to cover we will discuss 
separately.\vspace{0.3cm}\\
First suppose $\#(\Sigma_S\cap L)=2k+1$ is odd and one of these intersection
points is negative. Let $\nu L$ be a tubular neighborhood of $L$. The boundary
torus $\partial\nu L$ intersects $\Sigma_S$ in $2k+1$ pairwise disjoint, embedded
circles. Denote by $x_1,\dots,x_{2k+1}$ the intersection points of $L$ and 
$\Sigma_S$ where the indices are chosen such that, starting at $x_1$, traversing
along $L$ in the direction given by its orientation, we will meet $\Sigma_S$ in 
order of the intersection points. Without loss of generality we may assume that 
$x_1$ is a negative intersection. Each of the intersection points $x_i$ determines 
a meridian $\mu_i$ in $\partial\nu L$. These meridians in turn determine $2k+1$ 
cylinders $C_1,\dots,C_{2k+1}$ in $\partial\nu L$, where $\mu_j$ and $\mu_{j+1}$ 
should bound $C_j$. For each $j=1,\dots,k$, we remove from $\Sigma_S$ the disks
around $x_{2j}$ and $x_{2j+1}$ in which $\nu L$ and $\Sigma_S$ meet and glue in 
the cylinder $C_{2j}$. In this way we obtain an oriented, embedded surface
$\Sigma$ which meets $L$ negatively in one single intersection point. Observe, that
we obtained $\Sigma$ from the Seifert surface by adding $k$ one handles. Thus, the
group $\hfkhat(\sthree_0(K),L;\fraks)$ vanishes for $\spinc$-structures $\fraks$ which fulfill the following inequality.
\begin{equation}
  -\bigl<c_1(\fraks),[\Sigma]\bigr>
  >
  2g(\Sigma)-2=2sg(K)+2k-2.
 \label{eq:adj01}
\end{equation}
Since the orientation of $\Sigma$ and $\Sigma_S$ agree, the $\spinc$-structure $\fraks=-\lambda\cdot PD[\mu]$ for $\lambda\in\Z^+$ will fulfill
\begin{equation}
 -\bigl<c_1(\fraks),[\Sigma]\bigr>=2\lambda.
 \label{eq:adj02}
\end{equation}
Hence, the groups vanish for $\lambda<-sg(K)-k+1$. On the other hand, we may
use the surface $-\Sigma$ to say that the groups $\hfkhat(\sthree_0(K),-L;\fraks)$
vanish if
\begin{equation}
-\bigl<c_1(\fraks),-[\Sigma]\bigr>
  >
  2g(\Sigma)-2=2sg(K)+2k-2.
 \label{eq:adj03}
\end{equation}
In this case, for $\fraks=\epsilon\cdot PD[\mu]$ with $\epsilon\in\Z^+$ we have that
\begin{equation}
 -\bigl<c_1(\fraks),-[\Sigma]\bigr>
 =2\epsilon.
\label{eq:adj04}
\end{equation}
Thus, we have vanishing groups for $\epsilon>sg(K)+k-1$. According to the
point-swap symmetry given in Proposition~\ref{mysymmetry} we see that
\[
  \hfkhat(\sthree_0(K),L;\fraks-PD[L])
  \cong
  \hfkhat(\sthree_0(K),-L;\fraks).
\]
Observe, that $PD[L]=lk(K,L)\cdot PD[\mu]$. Thus, we see that 
$\hfkhat(\sthree_0(K),L;\fraks)$ vanish for
$\fraks>sg(K)+k-1-lk(K,L)$. Now, we just have to see that
\[
 k=
 \Bigl\lfloor
 \frac{\#(L\cap\Sigma_S)}{2}
 \Bigr\rfloor.
\]
Suppose that $\#(\Sigma\cap L)=2k+1$ is odd and all intersection points are 
positive. This case can be covered analogously as the first case. We, again, 
construct a surface $\Sigma$ as done above by $1$-handle additions such that
$L$ and $\Sigma$ intersect in a single point. In this case, to apply our adjunction 
inequality, we have to equip $\Sigma$ with the orientation from $-\Sigma_S$. We
then proceed as above: By the adjunction inequalities we know that
$\hfkhat(\sthree_0(K),L;\fraks)$ vanish for $\spinc$-structures $\fraks$ which 
fulfill the inequality $(\ref{eq:adj01})$: For $\fraks=\lambda\cdot PD[\mu]$ 
with $\lambda\in\Z^+$ (recall that $\Sigma$ carries the orientation from 
$-\Sigma_S$) we have that $(\ref{eq:adj02})$ is fulfilled. Hence, the groups
vanish for $\lambda>sg(K)+k-1$. Now, looking at $\hfkhat(\sthree_0(K),-L;\fraks)$
we see that these groups vanish for $\spinc$-structures $\fraks$ which fulfill
$(\ref{eq:adj03})$. So, for $\fraks=\epsilon\cdot PD[\mu]$ with $\epsilon\in\Z^-$ 
we have that $(\ref{eq:adj04})$ is fulfilled. This means that, by application of 
point-swap symmetry, the groups $\hfkhat(\sthree_0(K),L;\fraks)$ also vanish for 
$\fraks<-sg(K)+1-k-lk(K,L)$.\vspace{0.3cm}\\
The last case, when $L$ and $\Sigma_S$ intersect in an even number of points (or 
not at all) is the easiest case to cover. By the procedure of handle attachments
given at the beginning of the proof we may form a surface $\Sigma$ which is
disjoint from $L$, with its genus $g(\Sigma)=sg(K)+k$. Since $\Sigma$ and $L$ are 
disjoint, part $(b)$ of Theorem~\ref{myadjuncineq} applies, showing that the groups
$\hfkhat(\sthree_0(K),L;\fraks)$ vanish for $\spinc$-structures outside of the
set $I_1=[-sg(K)-k+1,sg(K)+k-1]$. The same holds for the groups 
$\hfkhat(\sthree_0(K),-L;\fraks)$. The latter, by point-swap symmetry, implies that
the groups $\hfkhat(\sthree_0(K),L;\fraks)$ vanish for $\fraks$ outside of
$I_2=[-sg(K)-k+1-lk(K,L),sg(K)+k-1-lk(K,L)]$. Thus, the groups vanish for $\fraks$ 
outside of the intersection $I_1\cap I_2$. In case the linking number of $K$ and 
$L$ is positive (or zero), $I_1\cap I_2$ equals the interval given in part $(a)$ 
and, in case the linking number is negative we obtain the interval in $(b)$.
\end{proof}
Suppose we are given a closed, oriented $3$-manifold $Y$ and a homologically
trivial knot $K$ in it. Further, denote by $\mu$ a meridian of $K$ and by
$\Fhat$ a closed surface in $Y_0(K)$ which is obtained by capping off a
Seifert surface of $F$ in $Y$ with the core of the surgery handle. Recall, that
the homology of the $0$-surgered manifold $Y_0(K)$ equals 
\[
 H^2(Y)\oplus\Z.
\]
The additional $\Z$-factor we obtain from the $0$-surgery is spanned by the
meridian $\mu$ and for a $\spinc$-structure $\fraks$ its value in 
this $\Z$ is given by the quantity
\[
 \bigl<c_1(\fraks),[\Fhat]\bigr>/2.
\]
With this preparation it is easy to see that the arguments presented in the 
proof of Proposition~\ref{thm:van} can be used to prove an analogue, but more general 
result, for $K\subset Y$ null-homologous and arbitrary $L\subset Y$ disjoint 
from $K$. As a special case of Proposition~\ref{thm:van} and knot conjugation symmetry we 
almost recover Theorem~1.1.~of Eftekhary's paper \cite{eftek}. 
Theorem~1.1.~states that the longitude Floer homology of
a knot in the $3$-sphere detects the Seifert genus of the knot in terms of
its filtration and that these groups fulfill some kind of conjugation symmetry. 
The theorem consists of four statements of whom we recover three.
\begin{proof}[Proof of Corollary~\ref{myres:lfh}] Given a knot $K$ in 
the $3$-sphere, Eftekhary's longitude Floer homology, by its definition, admits 
an identification
\begin{equation}
 \lfh(K,\fraks)=\hfkhat(\sthree_0(K),\mu;\fraks+1/2\cdot PD[\mu]).
\label{eq:ident}
\end{equation}
The $\spinc$-shift originates from the fact that Eftekhary filters the longitude 
Floer homology using the map
\[
\fraks=\frac{\fraks_z+\fraks_w}{2}
\]
instead of $\fraks_z$.
Since $\fraks_z=\fraks_w-PD[\mu]$ (see the discussion in the 
proof of Proposition~\ref{mysymmetry}), the identification follows. Given a
genus minimizing Seifert surface $\Sigma_S$ of $K$, this surface can be capped off
in $\sthree_0(K)$ with a disk. We obtain a closed surface $\Fhat$ and we have that
$\mu$ intersects $\Fhat$ in a single point, positively. So, part $(b)$ of
Proposition~\ref{thm:van} applies, telling us that the knot Floer homology vanishes
for $\spinc$-structures $\fraks$ outside of 
\[[-sg(K),sg(K)-1].\] With the
identification between the knot Floer and longitude Floer homology given in 
$(\ref{eq:ident})$, we see that the
latter vanishes for $\spinc$-structures outside of 
\[[-sg(K)+1/2,sg(K)-1/2].\] 
This concludes the first statement of the corollary. To prove the second statement,
recall that by knot conjugation symmetry we have an isomorphism
\[
 \mN\co
\hfkhat(\sthree_0(K),\mu;\fraks)
 \overset{\cong}{\lra}
 \hfkhat(\sthree_0(K),\mu;\fraks-1).
\]
Using the identification given in equation $(\ref{eq:ident})$ the statement follows.
\end{proof}
\begin{rem}\label{remark01} Recall that part $(i)$ of Theorem~\ref{myadjuncineq} 
gives an 
adjunction inequality without absolute values as given in part $(ii)$ of the
theorem. It is worth mentioning that absolute values in case of $(i)$ cannot
be achieved in general. The non-vanishing result part of Theorem~3.2.~of Eftekhary's
paper \cite{eftek} gives a contradiction. However, even without Eftekhary's result
we are able to see that $\hfkhat(\sthree_0(T),\mu)$ gives a counterexample. We will 
see that this group can be identified with $\hfkhat(T^3,K^*)$, where $K^*$ is a 
fiber of the fibration 
\[
  \sone\lra T^3\lra T^2.
\] 
Assuming that part $(i)$ of our theorem works with absolute values, this would 
mean that the former group (and hence the latter) is concentrated in the 
subgroup associated to the $\spinc$-structure with trivial first chern class. 
However, using the exact sequences given in \S\ref{computefoone} to 
compute the absolute $\Q$-gradings, this yields a contradiction to the fact 
that $T^3$ admits an orientation-reversing diffeomorphism that fixes the 
orientation of the knot $K^*$ (which implies a certain symmetry in 
these $\Q$-gradings).
\end{rem}

\section{Circle bundles over the sphere}\label{firstcalc}
We discuss the case of $\stwo\times\sone$ in this section as a first step in the
proof of Theorem~\ref{mythm:01}. To be more precise, starting from $\stwo\times\sone$ 
we will discuss all circle bundles whose base space is $\stwo$ by using surgeries
to change the Euler-number of the fibration.
\begin{definition}\label{fgedef} We define $\fge$ to be the $\sone$-bundle with 
Euler number $e$ and whose base space is a closed, oriented surface of genus $g$.
\end{definition}
\noindent We start reproving Proposition~\ref{oldres} (see also \cite[Proof of Theorem 7.4]{Saha}).
\begin{figure}[ht!]
\labellist\small\hair 2pt
\pinlabel {$\alpha$} [l] at 279 334
\pinlabel {$x$} [Bl] at 283 244
\pinlabel {$z$} [l] at 422 219
\pinlabel {$w$} [l] at 310 183
\pinlabel {$y$} [tl] at 284 121
\pinlabel {$\beta$} [r] at 146 44
\endlabellist
\centering
\includegraphics[width=6cm]{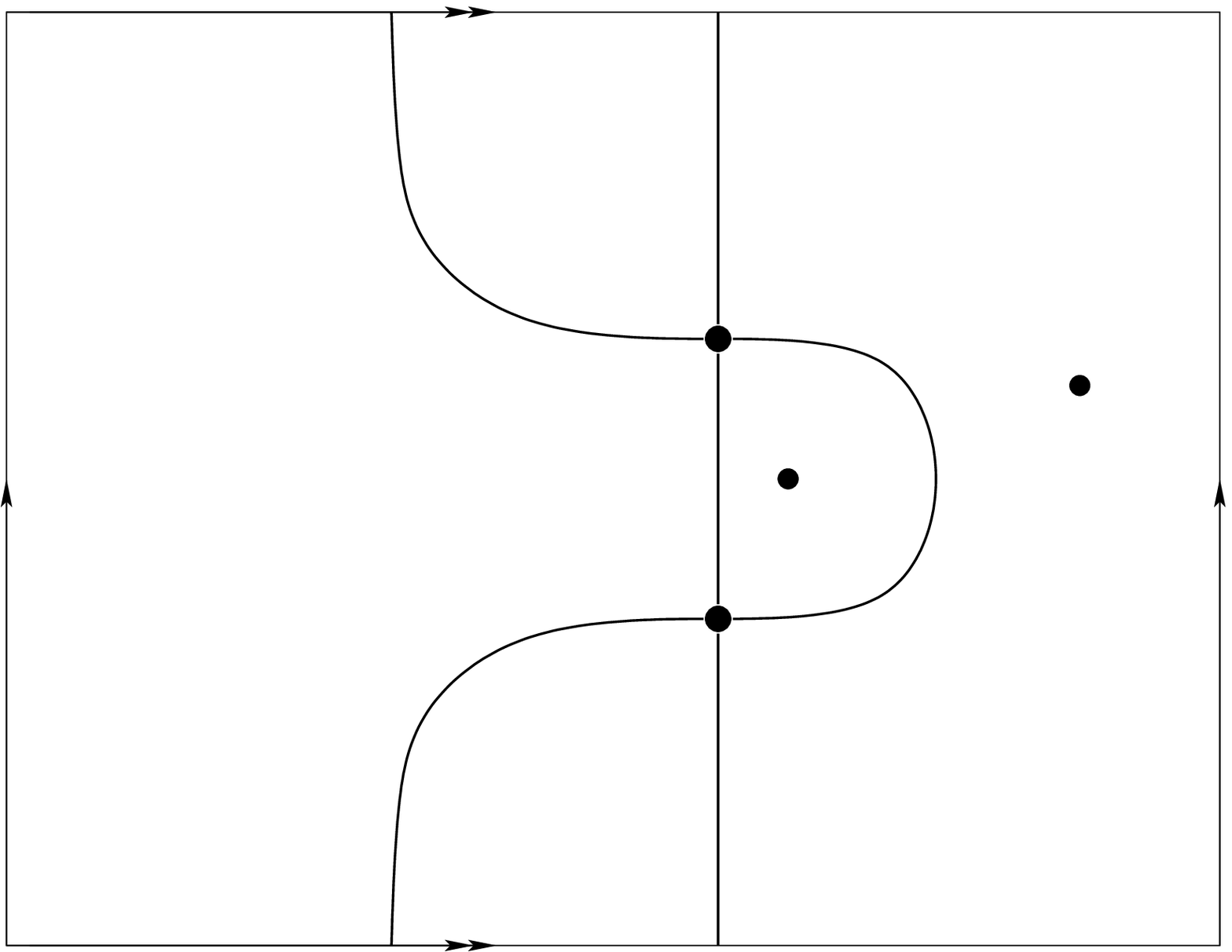}
\caption{Heegaard diagram adapted to $K^*$}
\label{Fig:modelknotfh}
\end{figure}

\begin{proof}[Proof of Proposition~\ref{oldres}] It is not hard to see 
that Figure \ref{Fig:modelknotfh} 
pictures a doubly-pointed Heegaard diagram $(T^2,\afat,\bfat,w,z)$
adapted to the pair $(\stwo\times\sone,K^*)$. The chain
complex $\cfkhat(\Sigma,\afat,\bfat)$ is generated by the intersection
points $\xfat$ and $\yfat$.  The boundary operator $\delhatw$ is given by:
\begin{eqnarray*}
  \delhatw \xfat
  &=&
  \yfat\\
  \delhatw \yfat
  &=&
  0.
\end{eqnarray*}
Thus, the result follows.
\end{proof}
For simplicity let us denote by $\fge$ the $\sone$-bundle over $\Sigma^g$ with 
Euler number $e$. Let $Y$ be a $\sone$-bundle over 
an oriented surface $\Sigma$. Choose a closed disk $D$ contained in a bundle chart. 
Then, by choosing sections $\sigma_1$ of $\left. Y\right|_D$ and $\sigma_2$ of
$\left. Y\right|_{\overline{\Sigma\backslash D}}$ we can define the intersection
number of $\sigma_1$ and $\sigma_2$ inside $\left. Y\right|_{\partial D}$. This
intersection number is the Euler number of the bundle (see~\cite{klaus}). Using this description it
is easy to see that starting with $\fgo$ it is possible
to change the Euler number of the bundle using certain surgeries along fibers:
Fix a fiber $\gamma$ in $\fgo$. This fiber projects to a point in $\Sigma$. Choose
a small disk around that point. Choose a constant section $\sigma_2$ over the
set $\overline{\Sigma\backslash D}$. This section restricted to 
$\left. Y\right|_{\partial D}\cong T^2$ is a meridian. A surgery along $\gamma$
with coefficient $-1/e$ provides the necessary modification. We obtain a
bundle with Euler number $e$.
\begin{prop}\label{stwosoneeuler} Denote by $\foe$ the $\sone$-bundle over $\stwo$ 
with Euler number $e$ and denote by $K^*$ a fiber.
Then $H_1(\foe)\cong\Z_{|e|}$ and for each $\spinc$-structure
$\fraks$ we have that $\hfkhat(\foe,K^*;\fraks)\cong\ztwo$.
\end{prop}
\begin{figure}[t]
\labellist\small\hair 2pt
\endlabellist
\centering
\includegraphics[width=10cm]{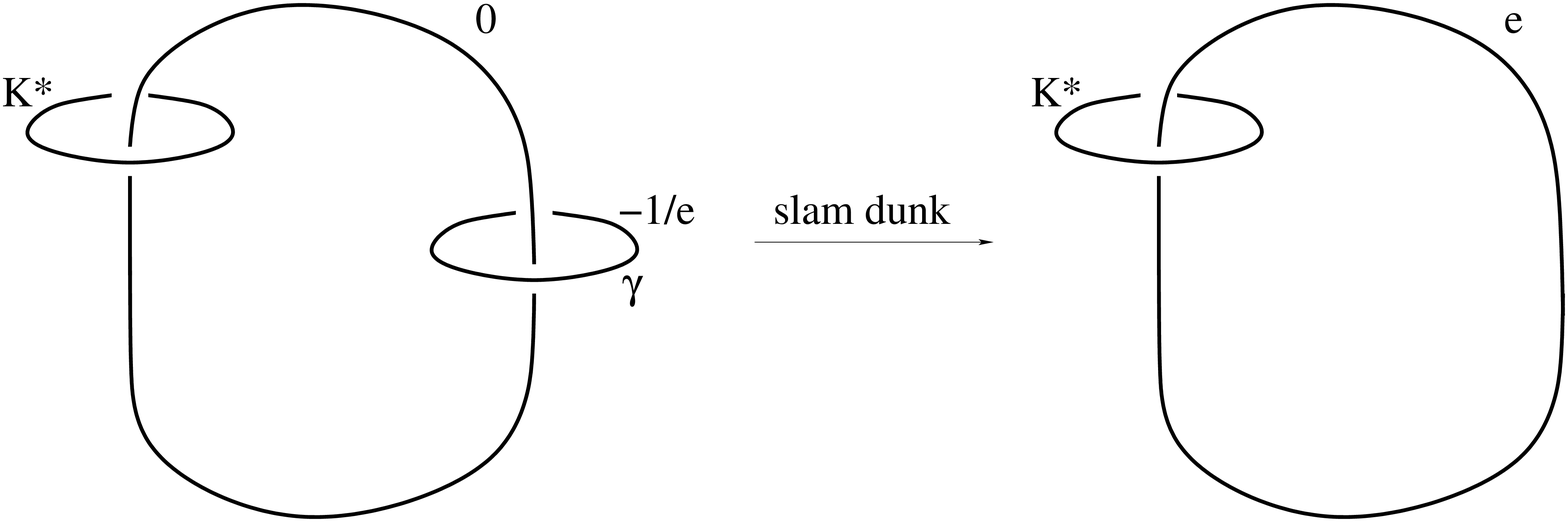}
\caption{Surgery raising the Euler number}
\label{Fig:firstproof}
\end{figure}

\begin{proof}  Denote by $\gamma$ and $K^*$ two distinct fibers of $\stwo\times\sone$. To change the Euler number we have to perform a $(-1/e)$-surgery
along $\gamma$. This is pictured in the left of Figure~\ref{Fig:firstproof}. By a 
slam dunk (cf.~\cite{GoSt}) we see that this is the same as the right of
Figure~\ref{Fig:firstproof}. Hence, $\foe$ is the lens space $L(1,e)$. As a matter
of fact, the knot $K^*$ in $L(1,e)$ generates the first homology group of $L(1,e)$.
To compute the knot Floer homology $\hfkhat(\foe,K^*)$ we have to find an adapted
Heegaard diagram: The $3$-sphere admits a Heegaard decomposition into two full-tori 
denoted by $H_1$ and $H_2$. It is possible to isotope the link pictured in the right 
of Figure~\ref{Fig:firstproof} such that the components of the link are the cores of 
the the full-tori $H_i$, $i=1,2$ where $K^*$ sits inside $H_1$. The gluing of the two
handlebodies $H_i$, $i=1,2$ is given by a diffeomorphism $\phi\co H_1\lra H_2$ which
sends $\mu_1$, a meridian of $H_1$, to $\lambda_2$, a longitude of $H_2$. Performing
the $e$-surgery as indicated in the right of Figure~\ref{Fig:firstproof}, we remove
a full-torus $D^2\times\sone$ from $H_2$ and glue a full-torus $H_2'$ back in, with a 
gluing map given by
\[
  \begin{array}{rrcl}
    \psi\co&\mu'&\lmt& e\cdot\mu_2+\lambda_2\\
    &\lambda'&\lmt&-\mu_2.
  \end{array}
\]
We form a new handlebody $H_1'$ by gluing together $H_1$ with 
$H_2\backslash(D^2\times\sone)$. We obtain a Heegaard decomposition 
$H_1'\cup_\partial H_2'$ of $\foe$ with gluing map given by
\[
  \mu_1'\lmt \psi^{-1}(\lambda_2)=\mu_2'+e\cdot\lambda_2'.
\]
This Heegaard decomposition is adapted to the knot $K^*$. 
Figure~\ref{Fig:firstproof02}
pictures the associated doubly-pointed Heegaard diagram. As we can see, there are
$e$ different intersection points generating the chain complex and all being 
associated to a different $\spinc$-structure. Thus, the differential vanishes identically and we have that
\[
 \hfkhat(\foe,K^*)=\hfhat(L(1,e))=\ztwo^{|e|}.
\]
\end{proof}

\begin{figure}[t]
\labellist\small\hair 2pt

\endlabellist
\centering
\includegraphics[width=6cm]{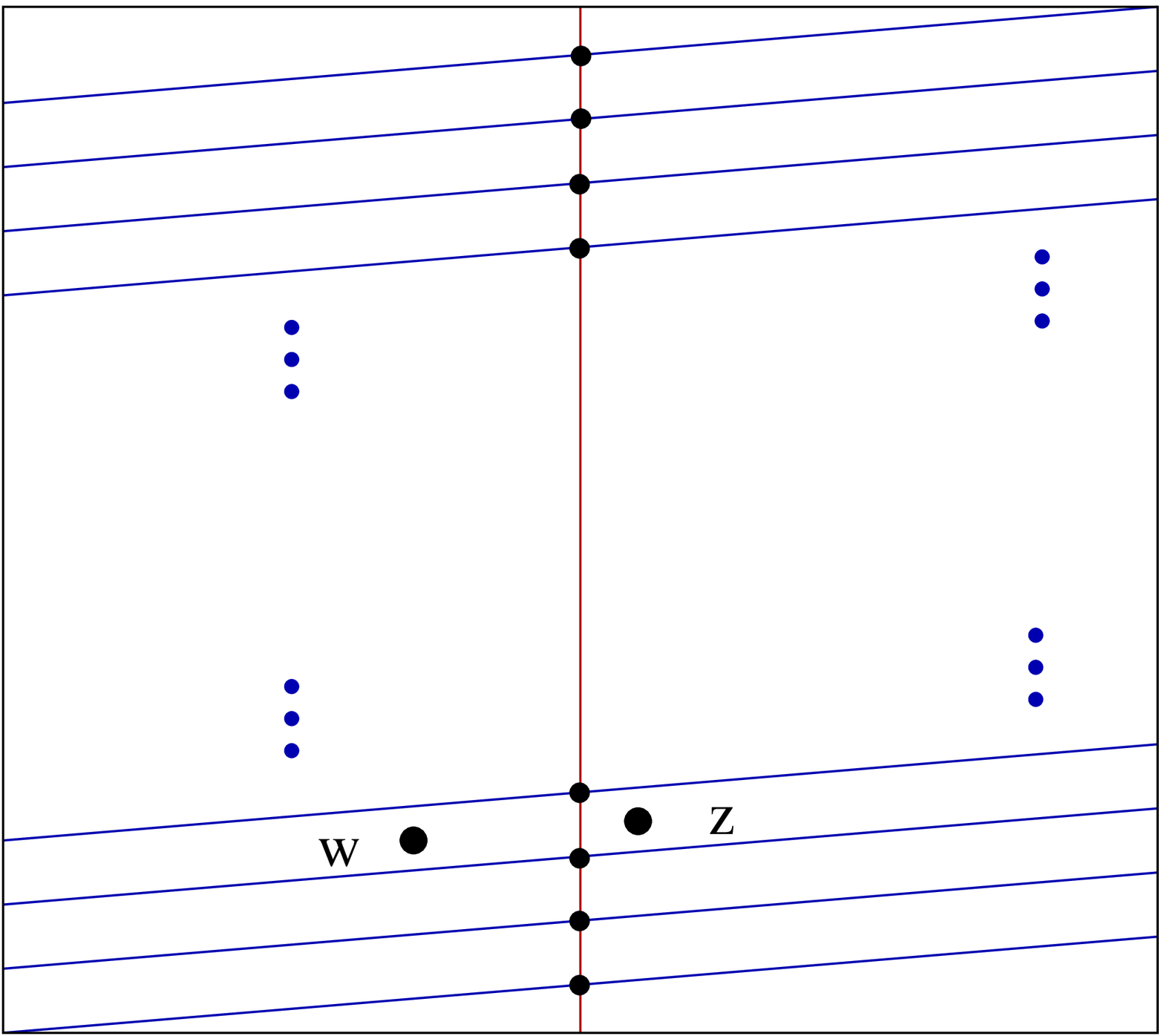}
\caption{Heegaard diagram adapted to the pair $(\foe,K^*)$.}
\label{Fig:firstproof02}
\end{figure}

\section{Three-dimensional Torus with a Fiber as Knot}\label{computefoone}
The computation given here serves as a demonstration how results 
like Proposition~\ref{oldres} can be used to simplify computations. 
In Figure~\ref{Fig:tthreeorient} we see a surgery diagram for $T^3$. Let us 
denote by $B_3$ the component linked with $K^*$, by $B_2$ the component linked with 
$B_3$ and by $B_1$ the remaining component.
 The manifold obtained from $\sthree$ by performing an 
$n$-surgery along $B_1$, an $m$-surgery along $B_2$ and a $k$-surgery along $B_3$ 
should be denoted by $B(n,m,k)$. The coefficients $n,m,k$ may lie in 
$\Z\cup\{\infty\}$. The surgery long 
exact sequence in knot Floer homology gives:
\begin{equation}
 \xymatrix@C=1.5pc@R=0.5pc{ 
 \hfkhat(B(0,\infty,0),K^*)\ar[r]
 &
 \hfkhat(B(0,0,0),K^*)\ar[r]
 &
 \hfkhat(B(0,1,0),K^*)\ar@/^1.5pc/[ll]
 \\ &&
 }\label{seq:comp01}
\end{equation}
The manifold $B(0,\infty,0)$ is a connected sum of two $\stwo\times\sone$'s. The
knot $K^*$ in it corresponds to a fiber of one of the $\stwo\times\sone$ components.
By \cite[Corollary~6.8]{Saha} we know that
\[
 \hfkhat(B(0,\infty,0),K^*)
 \cong
 \hfhat(\stwo\times\sone)
 \otimes
 \hfkhat(\stwo\times\sone,K^*).
\]
By Proposition~\ref{oldres} the group at the right is zero and, hence, the
tensor product vanishes. By exactness of the sequence we see that
\[
  \hfkhat(\mathbb{F}_1^0,K^*)
  =
  \hfkhat(B(0,0,0),K^*)
  \cong
  \hfkhat(B(0,1,0),K^*).
\]
\begin{figure}[t]
\labellist\small\hair 2pt
\endlabellist
\centering
\includegraphics[width=3cm]{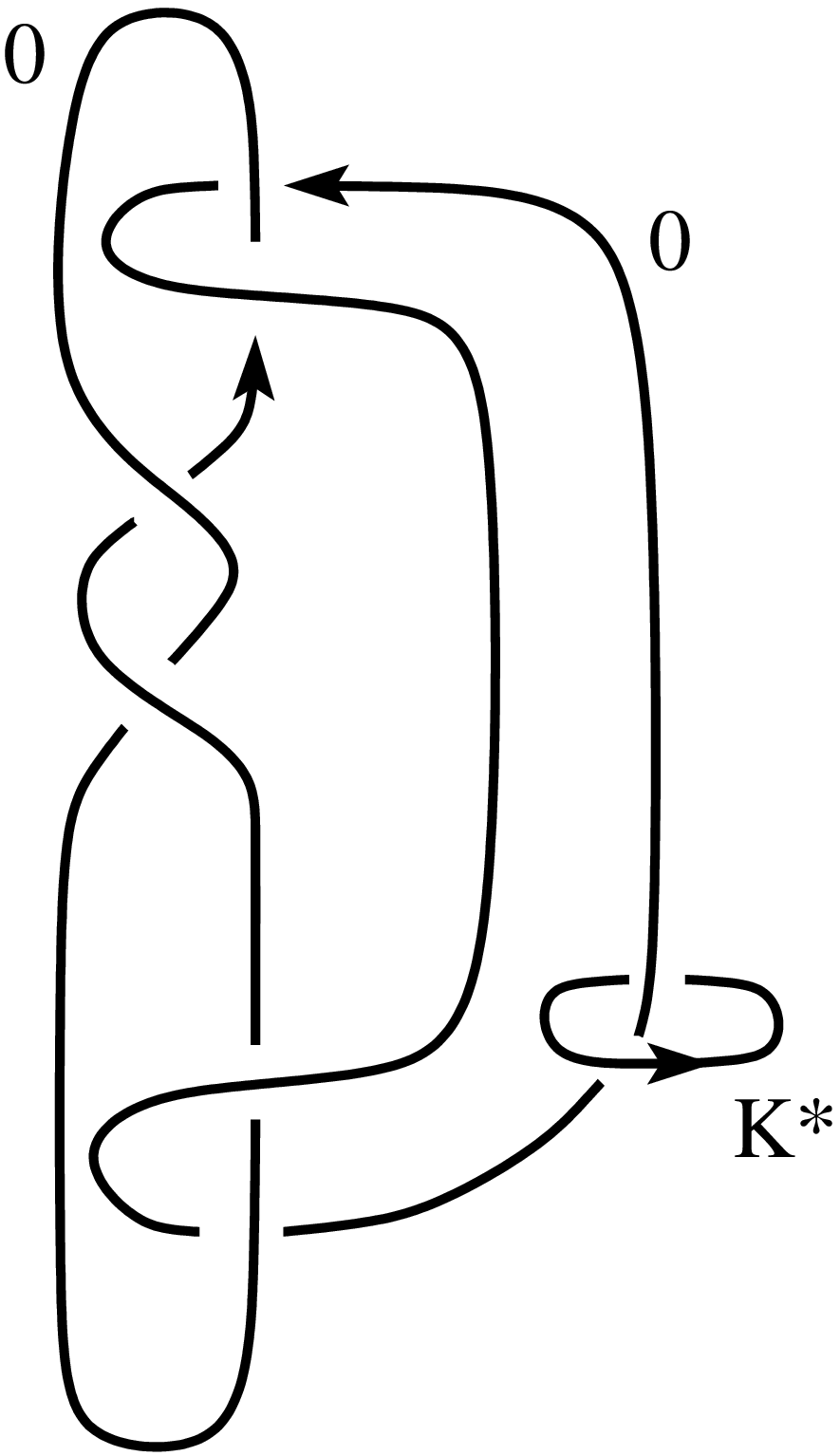}
\caption{$B(0,1,0)$ is isomorphic to the manifold given here.}
\label{Fig:whitehead}
\end{figure}
\noindent The manifold $B(0,1,0)$ is isomorphic to the manifold given in 
Figure~\ref{Fig:whitehead}. We kept track of the knot $K^*$. Denote by
$W(n,m)$ the manifold obtained from the link given in Figure~\ref{Fig:whitehead}
with surgery coefficient $n$ attached to the knot with the twist and with
surgery coefficient $m$ attached to the knot with meridian $K^*$. The following
sequence is exact:
\begin{equation}
  \xymatrix@C=1.5pc@R=0.5pc{
  \hfkhat(W(\infty,0),K^*)\ar[r]
  &
  \hfkhat(W(0,0),K^*)\ar[r]
  &
  \hfkhat(W(1,0),K^*)\ar@/^1.5pc/[ll]\\ &&
  }\label{seq:comp02}
\end{equation}
Since $(W(\infty,0),K^*)$ equals $(\stwo\times\sone,K^*)$ the corresponding knot 
Floer homology vanishes and, hence,
\[
  \hfkhat(W(0,0),K^*)
  \cong
  \hfkhat(W(1,0),K^*).
\]
Denote by $T$ the right-handed trefoil knot and denote by $\sthree_k(T)$ the result 
of a $k$-surgery along $T$. The pair $(W(1,0),K^*)$ is isomorphic to 
$(\sthree_0(T),\mu)$, where $\mu$ is a meridian of $T$ (see Figure~\ref{Fig:s30k}). 
By abuse of notation, we will denote this meridian by $K^*$. Consider the 
following exact sequence:
\begin{equation}
\xymatrix{\hfkhat(\sthree,T)\ar@/_2pc/[rr]^{\Goo}& 
\hfhat(\sthree_{+1}(T))\ar[l]_{\Foo} & \hfkhat(\sthree_0(T),K^*)\ar[l]_{\Hoo}\\ & 
& }
\label{seq:seqref01}
\end{equation}
This sequence is derived by either applying the Dehn twist sequence from 
\cite{Saha} or applying the surgery exact triangle in knot Floer homology to 
$(\sthree,T)$, where we start with a surgery along $\gamma$ with framing $0$
with $\gamma$ being a push-off of $T$ determining the $(+1)$-framing along $T$.
In both cases the map $\Fhat$ comes from counting holomorphic triangles in a
suitable Heegaard triple diagram. In the Dehn twist case this can be seen
by applying the results of \cite{Saha03}.
\begin{figure}[t]
\labellist\small\hair 2pt
\endlabellist
\centering
\includegraphics[width=3.5cm]{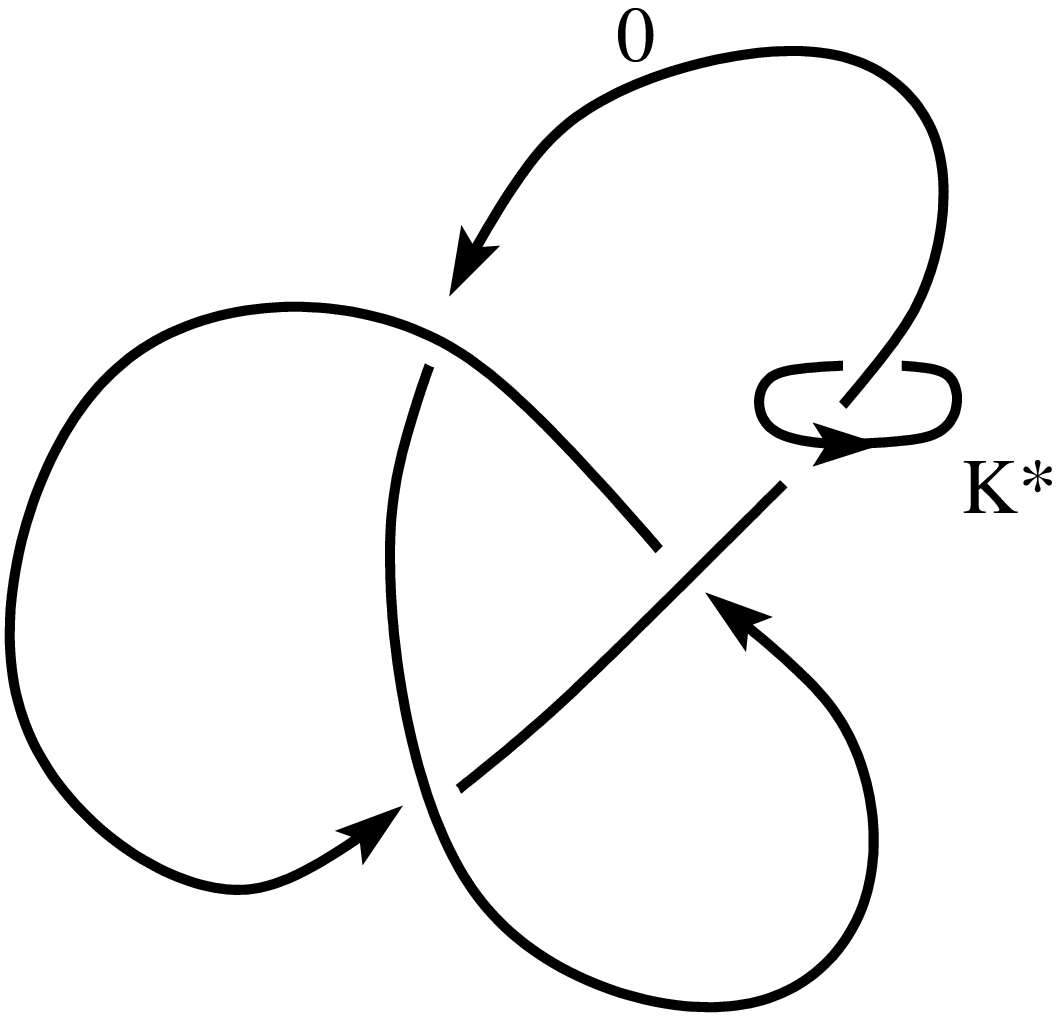}
\caption{$\hfkhat(T^3,K^*)$ is isomorphic to $\hfkhat(\sthree_0(T),K^*)$ 
with the orientations given in the picture.}
\label{Fig:s30k}
\end{figure}
\vspace{0.3cm}\\
It is worth mentioning, that our results on knot Floer homologies 
from the previous
section already reveals most of the structure of $\hfkhat(\sthree_0(T),\mu)$: 
First of all, Theorem~\ref{mysym:kfh} tells us that the rank of the knot Floer homology is even. It is easy to see that the right-handed trefoil has Seifert genus 
one. So, according to Proposition~\ref{thm:van}, the group is concentrated in the subgroups 
associated to the $\spinc$-structures $\fraks_0$ and $\fraks_0-PD[\mu]$, where
$\fraks_0$ is the $\spinc$-structure in $\sthree_0(T)$ with trivial first chern
class. The knot conjugation symmetry given in Corollary~\ref{knotconjug} provides an
isomorphism between the groups $\hfkhat(\sthree_0(T),\mu;\fraks_i)$, $i=1,2$. It 
remains to compute the rank of the homology.\vspace{0.3cm}\\
\noindent First of all $\hfhat(\sthree_{+1}(T))\cong\ztwo$ was calculated in \cite[\S8.1]{OsZa03} (cf.~also \cite{OzSti}). The 
knot Floer homology of the right-handed trefoil was given in \cite{OsZa04}, it is 
\[
  \hfkhat(\sthree,T)
  \cong
  (\ztwo)_{(-1;-2)}
  \oplus
  (\ztwo)_{(0;-1)}
  \oplus
  (\ztwo)_{(1;0)}
\]
where the subscript $(i;j)$ means that the corresponding factor sits in the subgroup
$\hfkhat_j(\sthree,T;i)$. Recall, that the maps defined by counting holomorphic
triangles, like $\Foo$, can be refined with respect to $\spinc$-structures of the 
underlying cobordism, $W$ say, i.e.
\[
 \Foo
 =
 \sum_{\fraks\in\mbox{\rm{\tiny Spin}}^c(W)}\!\!\!\!\!\!\!\!\Foo_\fraks
\]
where
\[
 \Foo_\fraks
 \co
 \hfhat(\sthree_{+1}(T);\left.\fraks\right|_{\sthree_{+1}(T)})
 \lra
 \hfkhat(\sthree,T;\left.\fraks\right|_{\sthree}).
\]
These maps shift the absolute $\Q$-grading. The
grading shift for $\Foo_\fraks$ is given by
\[
 d(\fraks)=\frac{1}{4}(c_1(\fraks)^2-2\chi(W)-3\sigma(W))
\]
To compute this shift of grading, we have to determine the topology of $W$ and
the signature of the intersection form. 
\begin{lem}\label{cobtopo} The Euler-characteristic 
of $W$ equals $1$ and the signature equals $-1$.
\end{lem}
\begin{proof} The cobordism $W$ is given by a $(-1)$-surgery along a meridian
$\gamma$ of a meridian $\mu$ of $T$. This means, that $W$ is given by attaching a single
$2$-handle to $\{1\}\times\sthree_{+1}(T)$ of $[0,1]\times\sthree_{+1}(T)$. Thus, 
$H_0(W,\partial_{-} W)
=
H_1(W,\partial_{-} W)
=
H_3(W,\partial_{-} W)
=
H_4(W,\partial_{-}W)=0$ 
and
$H_2(W,\partial_- W)\cong\Z$ (cf.~\cite[p.~111]{GoSt}).  
We compute 
the homology groups $H_k(W)$ using the long exact sequence of the pair 
$(W,\partial W)$ where $\partial W=-\sthree_{+1}(T)\sqcup\sthree$: 
We get that $H_0(W)=\Z$, $H_1(W)=0$, $H_2(W)=\Z$, $H_3(W)\cong\Z$ and $H_4(W)=0$.
Consequently, $\chi(M)$ equals $1$. We perform a positive Rolfsen twist about $\mu$.
This changes the framing of $T$ to $+1$, the framing of $\gamma$ to $0$ and links
$\gamma$ and $T$. In this picture, $\gamma$ is a meridian of $T$. Putting the
framing coefficient of $T$ in brackets gives a relative Kirby diagram of $W$. We define
a cobordism $X$ obtained from the $D^4$ by gluing a $2$-handle along $T$ with framing
$+1$. The cobordism
\[
 Z=X\cupb W
\]
is a handlebody and, thus, the intersection form $Q_{Z}$ is given by the linking
pairing of the Kirby diagram of $Z$ we obtain by joining the diagrams of $X$ and
$W$ (cf.~\cite[Proposition~4.5.11]{GoSt}). The linking matrix is given by
\begin{equation}
 \left(
 \begin{matrix}
 1 & 1 \\
 1 & 0
 \end{matrix}
 \right)\label{eq:inform}
\end{equation}
after choosing suitable orientations on $T$ and $\gamma$. The 
homology $H_2(Z)$ is generated by surfaces $S_1$ and $S_2$ we obtain by 
choosing Seifert surfaces bounding $T$ and $\gamma$ and capping these off 
with core discs of the handles we attach along $T$ and $\gamma$. Recall, 
that with respect to this basis the intersection form $Q_Z$
is given by (\ref{eq:inform}) (cf.~\cite[Proposition~4.5.11]{GoSt}). The
surface $S_1$ generates $H_2(X)$. The element $\alpha=(-1,1)$ is a primitive element
in $H_2(Z)$ such that $Q_Z(\alpha,[S_1])=0$. Hence, the element $\alpha$ generates 
$H_2(W,\partial W)\subset H_2(Z)$. It is easy to compute that $Q_Z(\alpha,\alpha)=-1$
and, hence, the signature of $W$ is $-1$.
\end{proof}
\noindent Applying the last lemma to the maps $\Fhat_\fraks$ the grading shift 
equals
\[
 d(\fraks)
 =\frac{1}{4}(c_1(\fraks)^2+1)
\]
and the values of $c_1(\fraks)^2$ are always less or equal to zero. Since, 
$\hfhat(\sthree_{+1}(T))$ is concentrated in degree $-2$ 
(cf.~\cite[\S8.1]{OsZa06}), we have that
\[
 \Foo
 =
 \sum_{\fraks\in\mbox{\rm{\tiny Spin}}^c(W)}\!\!\!\!\!\!\!\! \Foo_\fraks
 \;=
 \sum_{\fraks\in\mathcal{S}} \Foo_\fraks
\]
where $\mathcal{S}$ are those $\fraks\in\spinc(W)$ such that
$c_1(\fraks)^2=-1$. If $\frakt$ is not contained in $\mathcal{S}$, then
\[
 \Foo_\frakt\bigl(\hfhat(\sthree_{+1}(T))\bigr)
 \subset
 \hfkhat_{d(\frakt)-2}(\sthree,T)=0.
\]
Note that the group is zero, since $d(\frakt)-2\not\in\{-2,-1,0\}$. Recalling the 
knot conjugation invariance of the maps induced by cobordisms (see Proposition~\ref{knotconjug}), i.e.~
\[
 \Foo_{\fraks}
 =
 \mN_{(\sthree,T)}
 \circ
 \Foo_{\mN(\fraks)}
 \circ
 \mN_{(\sthree_{+1}(T),U)}
\]
where $U$ denotes the unknot, we see that $\Foo_{\fraks}$ vanishes if and only if 
$\Foo_{\mN(\fraks)}$ vanishes. However, in this case $\mN(\fraks)=\mJ(\fraks)$ since both $U$ and $T$ are null-homologous. Especially observe, that the 
$\spinc$-structure of $W$ with trivial chern class is not contained in 
$\mathcal{S}$. Thus, we can write $\mathcal{S}$ as
\[
 \{\fraks_{-1},\mJ(\fraks_{-1})\}
\]
where 
$c_1(\fraks_{-1})^2=c_1(\mJ(\fraks_{-1}))^2=-1$.
Thus,
\[
  \Foo
  =
  \Foo_{\fraks_{-1}}
  +
  \Foo_{\mJ(\fraks_{-1})}.
\]
With a Mayer-Vietoris computation we see that $H_1(\sthree_{+1}(T))=0$, 
since $H_1(\sthree\backslash\nu T)$ is generated by a meridian of $T$. Thus, 
\begin{eqnarray*}
 \hfhat(\sthree,T)
 &=&
 \hfhat(\sthree,T;\fraks_0)\\
 \hfkhat(\sthree_{+1}(T))
 &=&
 \hfkhat(\sthree_{+1}(T);\fraks_0)
\end{eqnarray*}
where $\fraks_0$ are the unique $\spinc$-structures on both $\sthree$ and 
$\sthree_{+1}(T)$ with vanishing first chern class. For these structures we have
that $\mJ(\fraks_0)=\fraks_0$ and, hence, 
$\Foo_{\fraks_{-1}}+\Foo_{\mJ(\fraks_{-1})}$ vanishes, i.e.~either the
summands vanish individually, or their sum vanishes. 
This finally shows that the map $\Foo$ vanishes identically and from exactness of
the sequence (\ref{seq:seqref01}) we get that
\[
 \hfkhat(\mathbb{F}^0_1,K^*)
 \cong
 \hfkhat(\sthree_0(T),K^*)\cong\ztwo^{4}.
\]
It is easy to see that
$H^2(T^3)\cong\Z\oplus\Z\oplus\Z$.  Recall, that we denoted by $B_1$ the component 
at the lower-left of 
Figure~\ref{Fig:t3}, $B_2$ the component at the top of Figure~\ref{Fig:t3}
and $B_3$ the component at the lower-right. We have chosen $K^*$ to be the meridian 
of $B_3$. The knot $B_1$ bounds a disk $D$ in $\sthree$ which intersects $B_3$ in 
exactly two points. Choose a small tubular neighborhood $\nu B_3$ of $B_3$, then $D$
cuts $\partial(\nu B_3)$ into two cylinders $Z_1$ and $Z_2$, i.e. $\partial(\nu 
B_3)\backslash{D}=Z_1\sqcup Z_2$. Pick one of them, $Z_1$ say, and define 
\[
  T
  =
  \overline{D\backslash\nu B_3}
  \cup_\partial
  \overline{Z_1}.
\]
$T$ is an orientable surface with boundary $B_1$. Endow $T$ with an arbitrary 
orientation. The surface $T$ can be thought of a sitting inside $T^3$. Here, the 
boundary $\partial T$ bounds a disk. We cap off $T$ with a disk to obtain $\That$. 
This torus $\That$ is disjoint from $K^*$. Observe, that the first 
homology of $T^3$ is generated by meridians $\mu_i$ of $B_i$, $i=1,\dots,3$. By 
abuse of notation we will also denote by $\mu_i$ the $\spinc$-structures  given by 
these meridians. By construction the genus of $\That$ is one and using the formula
$(\ref{eq:chernformula})$ we see that for any $\spinc$-structure 
$\fraks=\lambda\mu_1$ with $\lambda\not=0$ we have that
\[
 \bigl|\bigl<c_1(\fraks),[\That]\bigr>\bigr|
 =\bigl|2\lambda\cdot\#(\mu_1,\That)\bigr|
 =2\lambda
 >2g(\That)-2=0.
\]
By Theorem~\ref{myadjuncineq} we see that $\hfkhat(T^3,K^*;\fraks)$ is zero. Using 
the same construction as above we can construct a torus intersecting $\mu_2$ in a 
single point and which is disjoint from $K^*$. The same reasoning as above shows 
that the knot Floer homology groups are zero, too, for all 
$\fraks\in\Z\backslash\{0\}\bigl<\mu_2\bigr>$. A similar torus can be constructed 
using the knot $B_3$, however, in this case the torus intersects $K^*$ in one 
single point. By Theorem~\ref{myadjuncineq} we see that for all 
$\fraks\in\Z^+\backslash\{0\}\bigl<\mu_1\bigr>$ the 
associated knot Floer homology vanishes. 

\begin{theorem} For both $\spinc$-structures $\fraks_i$, $i=0,1$ the associated knot Floer homology $\hfkhat(T^3,K^*;\fraks_i)$ is isomorphic to $\ztwo\oplus\ztwo$. 
\end{theorem}
\begin{proof} Reversing the orientation of the fiber $K^*$, we may apply a 
similar reasoning as above to see that for all 
$\fraks\in(\Z\oplus\Z\oplus\Z^-)\backslash\{0\}$ the associated group 
$\hfkhat(T^3,-K^*;\fraks)$ is zero. Applying Proposition~\ref{mysymmetry} as it was 
done in the proof of Proposition~\ref{thm:van}, we see that
\[
 \hfkhat(T^3,K^*;\fraks)\not=0
\]
is possible if $\fraks\in\{\fraks_0,\fraks_0-PD[K]\}$, only. In the following, denote by $\fraks_1=\fraks_0-PD[K]$. With the point-swap symmetry 
(see Proposition~\ref{mysymmetry2}) we 
see that both groups $\hfkhat(T^3,K^*;\fraks_0)$ and $\hfkhat(T^3,K^*;\fraks_1)$ 
are isomorphic. So, both $\hfkhat(T^3,K^*;\fraks_i)$ for $i=0,1$ are isomorphic 
to $\ztwo^2$. 
\end{proof}
\section{Proof of Theorem~\ref{mythm:01}}\label{sec:surgerypicture}
\begin{figure}[t]
\labellist\small\hair 2pt
\endlabellist
\centering
\includegraphics[width=5cm]{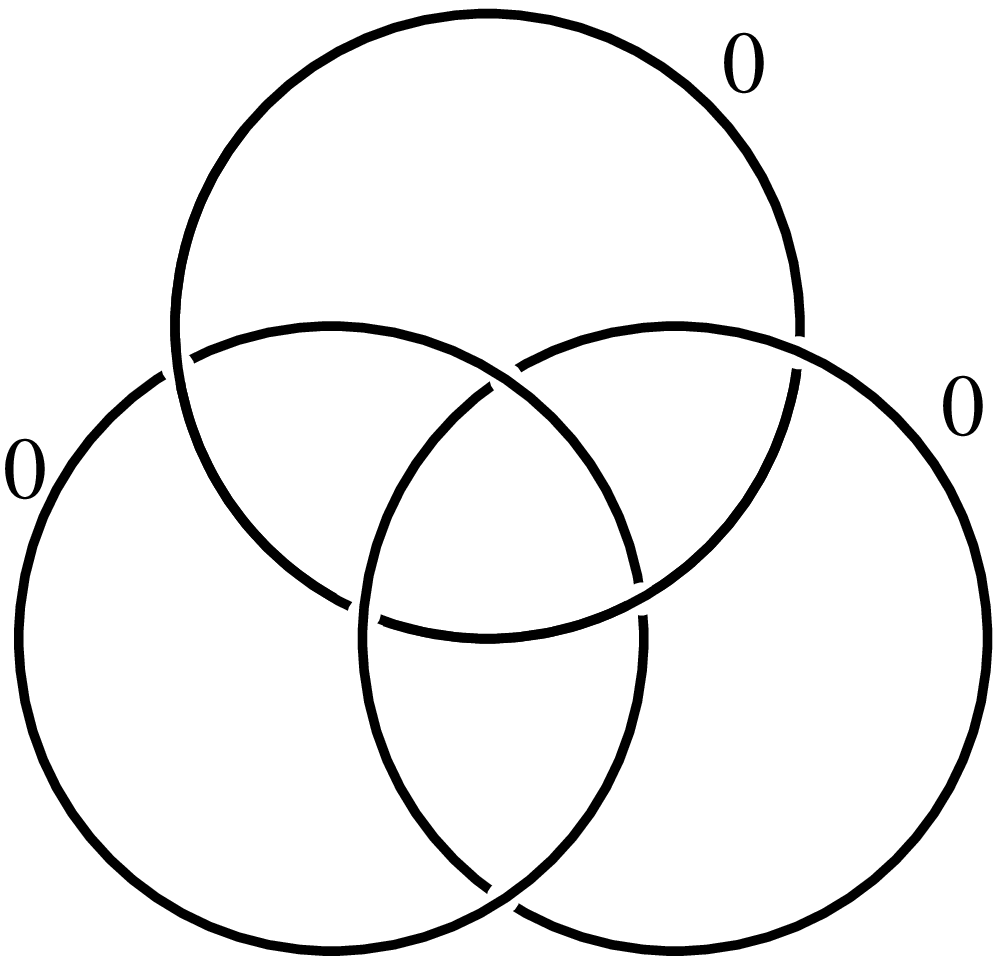}
\caption{Surgery diagram for $T^3$.}
\label{Fig:t3}
\end{figure}
In Figure~\ref{Fig:t3} we see a surgery diagram for the $3$-dimensional torus $T^3$.
The pictured link is called the Borromean link. To describe the bundle 
$\Sigma_2\times\sone$ we will start with two $3$-dimensional tori and perform a 
fiber-connected sum. It is not hard to see, that a fiber-connected sum of two
$3$-dimensional tori yields $\Sigma_2\times\sone$. Denote by $B\subset\sthree$ the 
Borromean link as pictured in Figure~\ref{Fig:t3}. Denote the three components of 
$B$ by $B_1$, $B_2$ and $B_3$. Observe, that a meridian of $B_3$ say is isotopic to a 
fiber of $T^3$. Now decompose the sphere $\sthree$ into two full-tori $H_0$ and 
$H_1$, where the gluing between these two tori is given by a map $\phi\co\partial 
H_0\lra\partial H_1$ sending a meridian $\mu_0$ of $\partial H_0$ to a longitude 
$\lambda_1$ of $\partial H_1$ and sending a longitude $\lambda_0$ of $\partial H_0$ 
to a meridian $\lambda_1$ of $\partial H_1$. It is possible to isotope the Borromean 
link $B$ such that it sits in the interior of $H_0$ and such that $B_3$ is the core 
of the torus $H_0$. In this description $\mu_0$, and hence $\lambda_1$, are fibers 
of the torus $T^3$. Thus, we may think the fibration of $T^3$ at 
$H_1=D^2\times\sone$ as being given
by the projection
\[
  \begin{array}{rrclcc}
  H_1=&D^2&\times&\sone&\lra&\sone\\
      &(r&,&\theta)&\lmt&\theta.
  \end{array}
\]
Now, consider a second $3$-dimensional torus. This torus is given by surgery along
a Borromean link $B'\subset\sthree$. We again decompose $\sthree$ into two 
tori $H_0'$ and $H_1'$ but now isotope $B'$ to sit in the interior of $H_1'$, the 
component $B_3'$ being the core of $H_1'$. Analogous to our description before, we 
denote by $\mu_i'$, $i=0,1$ meridians of $H_i'$ and by $\lambda_i'$, $i=0,1$ 
longitudes of $H_i'$. The gluing of the tori $H_0'$ and $H_1'$ is given like 
described above for $H_0$ and $H_1$. We may think the fibration of the 
$3$-dimensional torus at $H_0'=D^2\times\sone$ as being given by the projection
\[
  \begin{array}{rrclcc}
  H_0'=&D^2&\times&\sone&\lra&\sone\\
      &(r&,&\theta)&\lmt&\theta.
  \end{array}
\]
A fiber-connected sum of the tori can be described by 
\begin{equation}
  T^3\backslash(int(H_1))\cup_\partial T^3\backslash(int(H_0'))
  \label{eq:fibersum}
\end{equation}
where the gluing is given by identifying $\lambda_1$ with $\lambda_0'$. Since
$\lambda_1$ is identified with $\mu_0$ and $\lambda_0'$ identified with $\mu_0'$ 
the manifold (\ref{eq:fibersum}) turns into
\begin{equation}
 H_0\cup_{\partial}H_1'
 \label{eq:fibersum2}
\end{equation}
where the gluing is given by sending $\mu_0$ to $\mu_1'$. Thus, $\Sigma_2\times\sone$
is obtained from $\stwo\times\sone$ by performing $0$-surgery along two Borromean
knots. The manifold $\stwo\times\sone$ is given by performing a $0$-surgery along
an unknot in $\sthree$. Observe, that in (\ref{eq:fibersum2}) the knot $B^3$ 
intersects each sphere $\stwo\times\{*\}$ in exactly one point, transversely. In the
surgery picture of $\stwo\times\sone$ we can exactly determine the spheres. In 
Figure~\ref{Fig:fibersandknot} we see the intersection the unit circle in 
$\R^2\times\{0\}\subset\R^3$, sitting in the $xy$-plane, with the $xz$-plane. Assume 
we have performed a $0$-surgery along that knot. In Figure~\ref{Fig:fibersandknot} 
the red lines picture the intersections of the spheres $\stwo\times\{*\}$ with the 
$xz$-plane. The knot $B_3$ is a knot which intersects each of the spheres in a 
single point, transversely. The blue circle around the left point pictures an 
appropriate knot. Thus, we may think of $B_3$ as being that knot. Since $B'_3$ has 
intersects each of
the $\stwo\times\{*\}$ in a single point, too, the blue circle around the right 
point of Figure~\ref{Fig:fibersandknot} can be thought of as begin $B'_3$. Hence, 
Figure~\ref{Fig:surgerysigmatwo} is a surgery diagram for $\Sigma_2\times\sone$. 
\begin{figure}[t]
\labellist\small\hair 2pt
\endlabellist
\centering
\includegraphics[width=10cm]{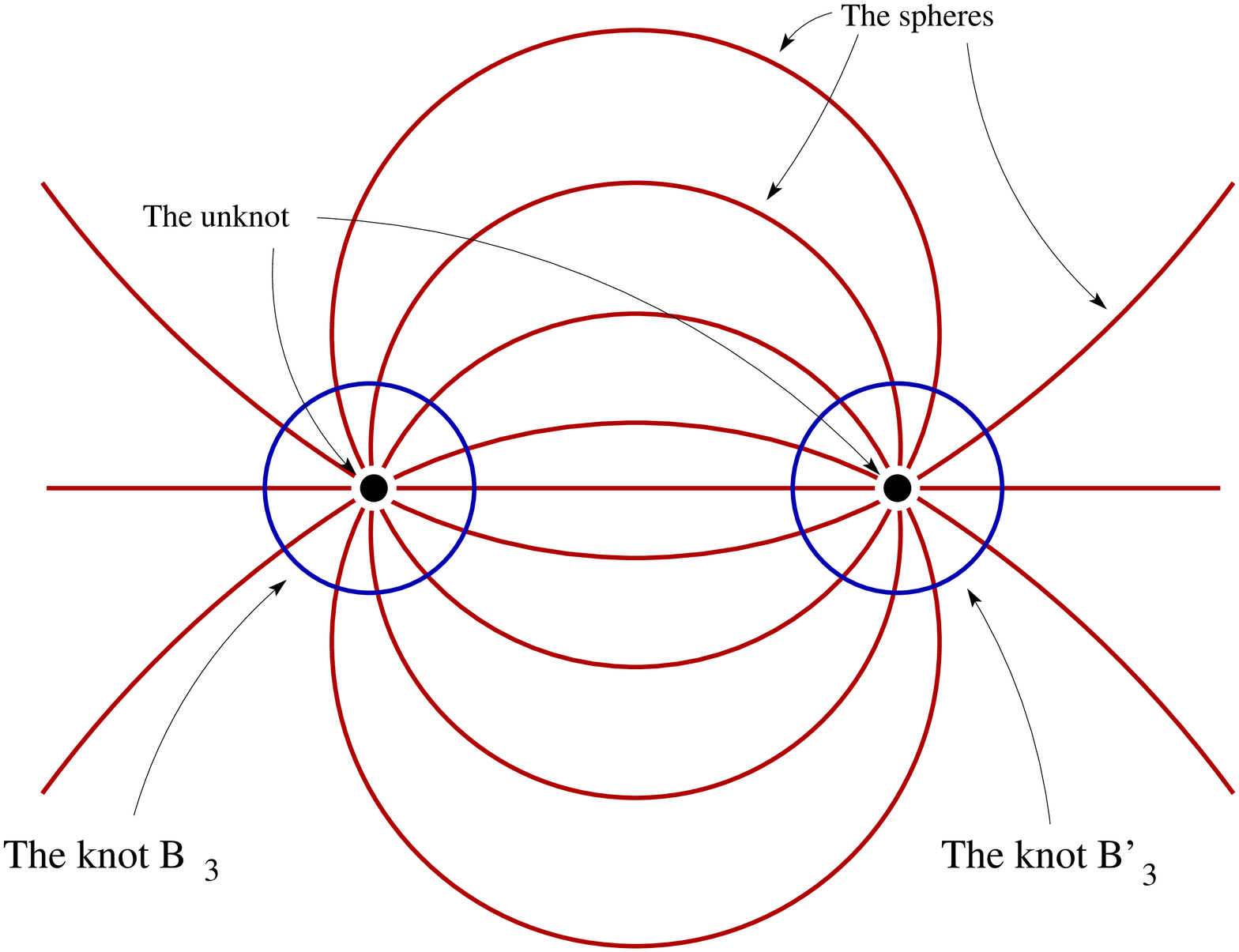}
\caption{A picture of $\stwo\times\sone$ with its spheres $\stwo\times\{*\}$ pictured in red.}
\label{Fig:fibersandknot}
\end{figure}
We may inductively continue this process to generate surgery diagrams for each
bundle $\Sigma_g\times\sone$. Applying a couple of handle slides and slam dunks 
we obtain the surgery presentation of $\fge$ (cf.~Definition~\ref{fgedef}) pictured 
in Figure~\ref{Fig:surgerypicture}.
\begin{figure}[t]
\labellist\small\hair 2pt
\endlabellist
\centering
\includegraphics[width=8cm]{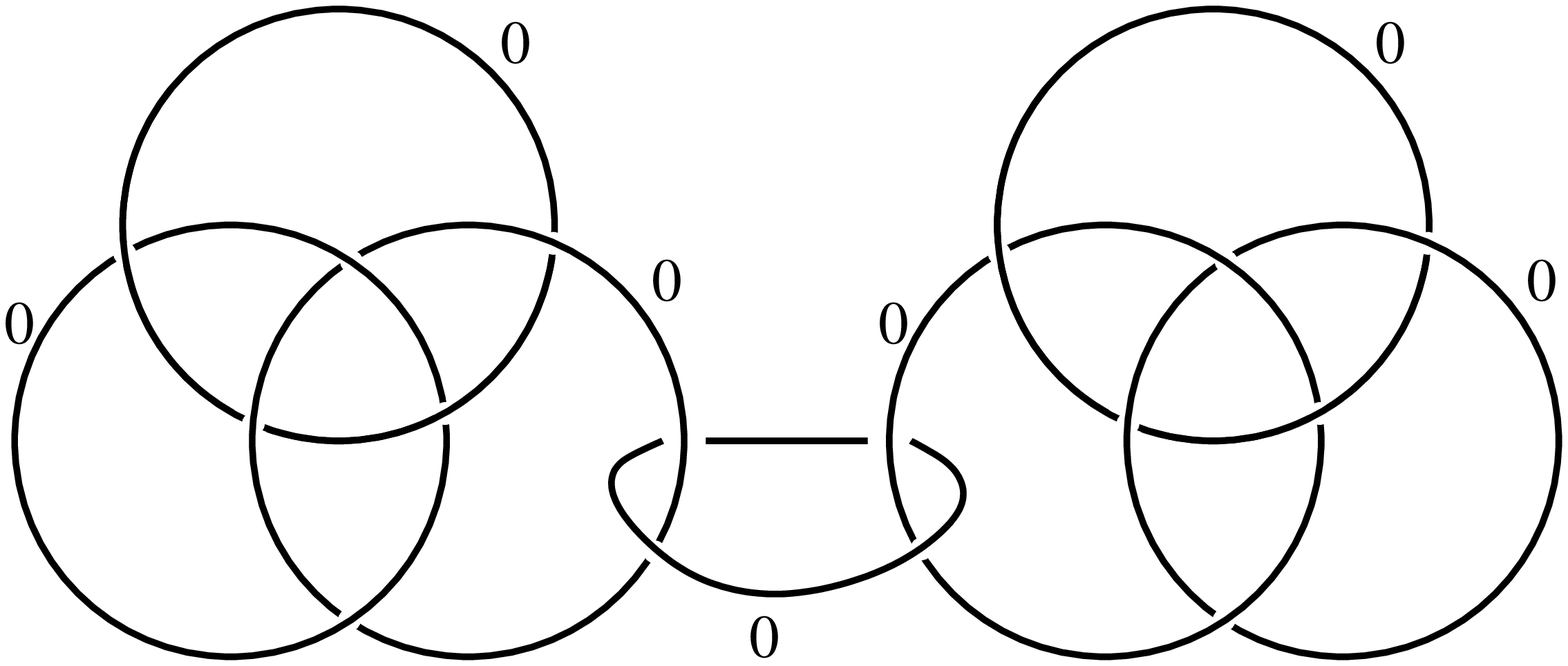}
\caption{A surgery diagram for $\Sigma_2\times\sone$.}
\label{Fig:surgerysigmatwo}
\end{figure}

\begin{proof}[Proof of Theorem~\ref{mythm:01}] 
In the following, $K^*$ will denote a fiber 
of $\fge$ (cf.~Definition~\ref{fgedef}).
In \S\ref{firstcalc} we have shown by explicit calculations, that
$\hfkhat(\fge,K^*)$ is non-zero for $g=0$ and arbitrary non-zero Euler 
number and checked that it vanishes completely for $g=e=0$. It remains
to show that the theorem holds for $g\geq1$ and $e\in\Z$:
\begin{figure}[t]
\labellist\small\hair 2pt
\endlabellist
\centering
\includegraphics[height=5.5cm]{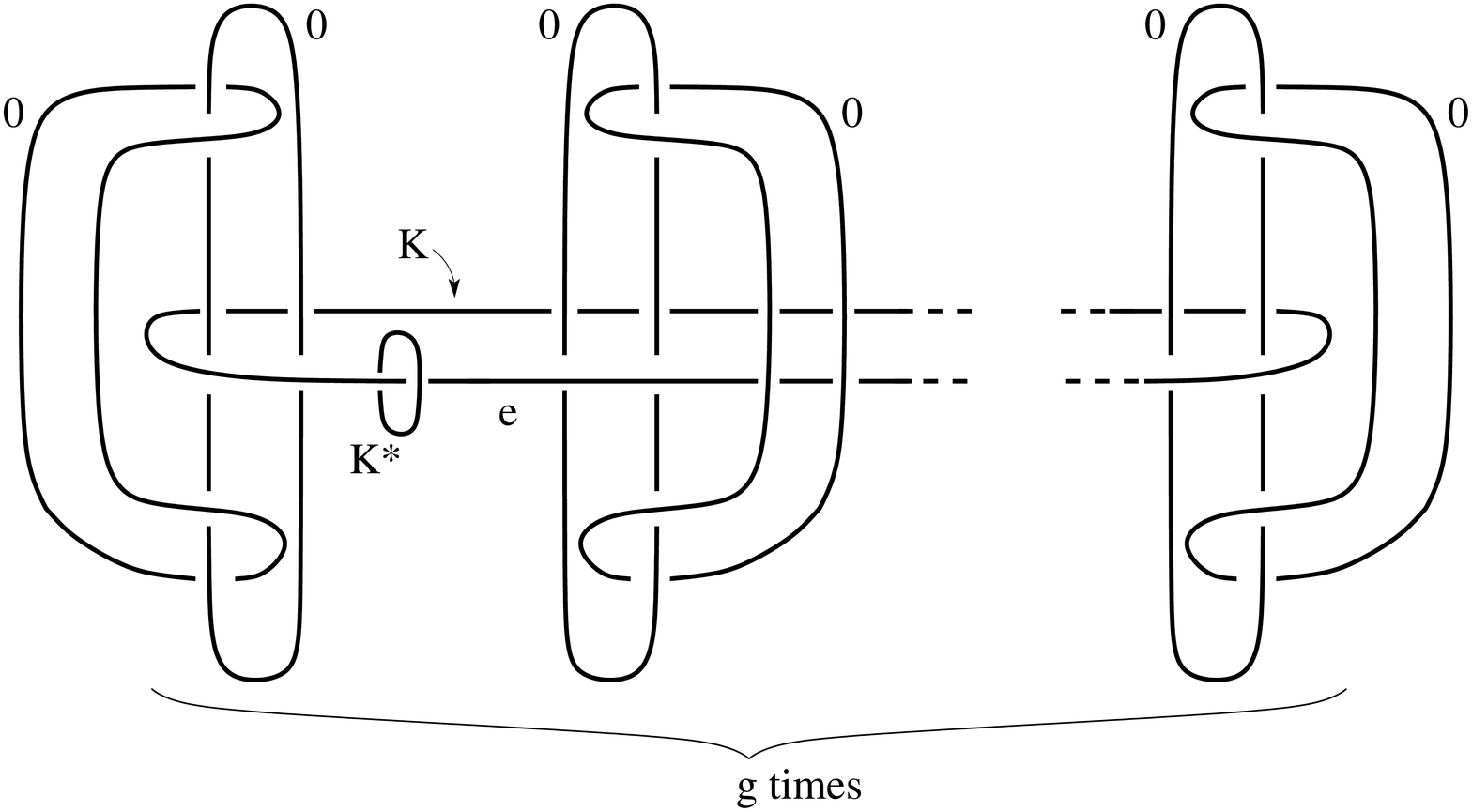}
\caption{A surgery diagram for $\fge$ with knot $K^*$.}
\label{Fig:surgerypicture}
\end{figure}
Starting from $Y=\#^{2g}(\stwo\times\sone)$ we perform an $e$-surgery along $K$ to obtain
$\fge$. We denote by $W$ the induced cobordism. 
The manifolds $Y$ and $\fge$ fit into the following surgery exact triangle.
\begin{equation}
 \xymatrix@C=1.5pc@R=0.5pc
{
\hfkhat(Y,U)\ar[r]^{\Foo_W}
&
 \hfkhat(\fge,K^*)\ar[r]
&
 \hfkhat(\fgepo,K^*)\ar^{\Hoo}@/^1.5pc/[ll]\\ &&
}\label{seq:doit2}
\end{equation}
where $U$ denotes the unknot in $Y$. We start with a careful investigation 
of $\Fhat$ in case $e$ is even: It is 
easy to see that $H_2(Y)\cong\Z^{2g}$ and that 
$H_2(\fge)\cong\Z^{2g}\oplus\Z_e$, where we introduced the convention 
that $\Z_1=\Z=\Z_0$. Using adjunction inequalities we see that $\hfkhat(Y,U)$ 
is concentrated in $\fraks_0$ and that $\hfkhat(\fge,K^;\fraks)$ can be 
non-zero for $\fraks\in\Z_e$, only. First observe, that 
\[
 \fraks\not=\mN(\fraks)
\]
for all $\fraks\in\spinc(W)$ which extend the structures $\fraks_0$ on $Y$ and
$\fraks_*\in\Z_e$ on $\fge$: If this were true for a $\spinc$-structure $\fraks$, 
this would imply that
\[
 -\left.\fraks\right|_{\fge}+PD[K]=\left.\mN(\fraks)\right|_{\fge}
 =
 \left.\fraks\right|_{\fge}
\]
which especially means that $[K]$ can be divided by two. This is not possible since
$e$ is even. Now observe, that for $\fraks\in\spinc(W)$ extending $\fraks_0$ on $Y$ 
we have that
\[
  \Foo_{W,\fraks}
  =
  \mN_{\fge}
  \circ
  \Foo_{W,\mN(\fraks)}
  \circ
  \mN_{Y}
  =
  \mN_{\fge}
  \circ
  \Foo_{W,\mN(\fraks)}.
\]
Fix a $\spinc$-structure $\fraks_*$ in the $\Z_e$-part of $H_2(\fge)$. Then denote 
by $\mS$ the set of $\spinc$-structures of $W$ that extend $\fraks_0$ on $Y$ 
and $\fraks_*$ on $\fge$. 
The map $\Foo_{W}$ restricted to $\hfkhat(\fge,K^*;\mN(\fraks_*))$
can be written as
\[
 A=\sum_{\fraks\in\mS}\Foo_{W,\mN(\fraks)}.
\]
The map $\Foo_{W}$ restricted to $\hfkhat(\fge,K^*;\fraks_*)$
can be written as
\[
 \sum_{\fraks\in\mS}\Foo_{W,\fraks}=\mN\circ A.
\]
Furthermore, we know that there is an isomorphism
\[
 \hfkhat(\fge,K^*;\fraks)\cong\hfkhat(\fge,K^*;\mN(\fraks)).
\]
Thus, the map $\Foo_{W}$ restricted to the sum
\[
 \hfkhat(\fge,K^*;\fraks)\oplus\hfkhat(\fge,K^*;\mN(\fraks))
\]
equals the map 
\[
  \mN\circ A\oplus A
  \co
  \hfkhat(Y,U;\fraks_0)
  \lra
  \hfkhat(\fge,K^*;\fraks)\oplus\hfkhat(\fge,K^*;\mN(\fraks))
\] 
with the property that $A(x)=0$ if and only if $\mN\circ A(x)=0$. Thus, 
$\mN\circ A\oplus A$ cannot be surjective. Consequently, $\Foo_{W}$ cannot 
be surjective unless the group $\hfkhat(\fge,K^*)$ is zero, completely. 
In that case, the rank of $\hfkhat(\fgepo,K^*)$ is non-zero by exactness 
of the sequence. Otherwise, since $\Foo_{W}$ is not surjective, it cannot 
be an isomorphism and, thus, $\hfkhat(\fgepo,K^*)$ is non-zero. Our 
considerations were independent of $g$, we just required that $e$ was 
even. So, we have shown that the rank of $\hfkhat(\fge,K^*)$ for $e$ 
odd cannot be zero.\vspace{0.3cm}\\
Now suppose that $e$ is odd. This means that $e+1$ is even. We want to see, that 
$\Hoo$ cannot be injective: Denote by $W'$ the cobordism associated to the map $\Hoo$.
Here, again, the map
\begin{equation}
  \mN
  \co
  \spinc(W')
  \lra
  \spinc(W')
  \label{spinmap:fpfree}
\end{equation}
has no fixed points. Given a $\spinc$-structure $\fraks_*$ in 
the $\Z_e$-part of $H_1(\fgepo)$ and denote by $\mS$ the set 
of $\spinc$-structures on $W'$ that extend $\fraks_*$ and $\fraks_0$
on $Y$. We see that
\[
  \left.\Hoo
  \right|_{\widehat{\mbox{\rm\tiny HFK}}
  \scriptstyle(\fgepo,K^*;\mN(\fraks_*))}
  =
  \sum_{\fraks\in\mS}\Hoo_{\mN(\fraks)}
  =:A.
\]
Similarly, since $(\ref{spinmap:fpfree})$ is without fixed points, the we have the following
equality
\[
  \left.\Hoo\right|_{\widehat{\mbox{\rm\tiny HFK}}\scriptstyle(\fgepo,K^*;\fraks_*)}
  =
  \sum_{\fraks\in\mS}\Hoo_{\fraks}=A\circ\mN.  
\]
Thus, the map restricted to the groups associated to the $\spinc$-structures 
$\fraks_*$ and $\mN(\fraks_*)$ equals
\[
 A\circ\mN+A
 \co
 \hfkhat(\fgepo,K^*;\fraks_*)
 \oplus
 \hfkhat(\fgepo,K^*;\mN(\fraks_*))
 \lra
 \hfkhat(Y,U;\fraks_0).
\]
This map cannot be injective, since for a non-zero $g\in\hfkhat(\fgepo,K^*;\fraks)$
the element $g\oplus -\mN(g)$ goes to zero, unless the group $\hfkhat(\fgepo,K^*)$ 
vanishes, completely. If the group vanishes completely, by exactness of the sequence, 
we know that $\hfkhat(\fge,K^*)$ is non-zero. In the non-vanishsing case, however, 
$\Hoo$ is not injective and, thus, by exactness of the sequence $\hfkhat(\fge,K^*)$
is non-zero. This proves Theorem~\ref{mythm:01}.
\end{proof}

\bibliographystyle{amsplain}

\providecommand{\bysame}{\leavevmode\hbox to3em{\hrulefill}\thinspace}
\providecommand{\MR}{\relax\ifhmode\unskip\space\fi MR }
% \MRhref is called by the amsart/book/proc definition of \MR.
\providecommand{\MRhref}[2]{%
  \href{http://www.ams.org/mathscinet-getitem?mr=#1}{#2}
}
\providecommand{\href}[2]{#2}

\end{document}